\pdfoutput=1
\RequirePackage{ifpdf}
\ifpdf 
\documentclass[pdftex]{sigma}
\else
\documentclass{sigma}
\fi

\usepackage[all]{xy}

\newtheorem{Th}{Theorem}[section]
\newtheorem{Cor}[Th]{Corollary}
\newtheorem{Prop}[Th]{Proposition}
\newtheorem{Lem}[Th]{Lemma}
{ \theoremstyle{definition}
\newtheorem{Def}[Th]{Definition}
\newtheorem{Rem}[Th]{Remark}
\newtheorem{Ex}[Th]{Example}}

\begin{document}


\newcommand{\arXivNumber}{1503.01529}

\renewcommand{\PaperNumber}{081}

\FirstPageHeading

\ShortArticleName{Monge--Amp{\`e}re Systems with Lagrangian Pairs}

\ArticleName{Monge--Amp{\`e}re Systems with Lagrangian Pairs}

\Author{Goo ISHIKAWA~$^\dag$ and Yoshinori MACHIDA~$^\ddag$}

\AuthorNameForHeading{G.~Ishikawa and Y.~Machida}

\Address{$^\dag$~Department of Mathematics, Hokkaido University, Sapporo 060-0810, Japan}
\EmailD{\href{mailto:ishikawa@math.sci.hokudai.ac.jp}{ishikawa@math.sci.hokudai.ac.jp}}

\Address{$^\ddag$~Numazu College of Technology, 3600 Ooka, Numazu-shi, Shizuoka, 410-8501, Japan}
\EmailD{\href{mailto:machida@numazu-ct.ac.jp}{machida@numazu-ct.ac.jp}}

\ArticleDates{Received April 10, 2015, in f\/inal form October 05, 2015; Published online October 10, 2015}

\Abstract{The classes of Monge--Amp{\`e}re systems, decomposable and
bi-decomposable Monge--Amp{\`e}re systems,
including equations for improper af\/f\/ine spheres and hypersurfaces of constant Gauss--Kronecker
curvature are introduced. They are studied by the clear geometric setting of
Lagrangian contact structures, based on the existence of Lagrangian pairs
in contact structures.
We show that the Lagrangian pair is uniquely determined by such a~bi-decomposable system
up to the order, if the number of independent variables $\geq 3$.
We remark that, in the case of three variables,
each bi-decomposable system is generated by a~non-degenerate three-form
in the sense of Hitchin.
It is shown that several classes of homogeneous Monge--Amp{\`e}re systems with Lagrangian pairs
arise naturally in various geometries.
Moreover we establish the upper bounds on the symmetry dimensions of
decomposable and bi-decomposable Monge--Amp{\`e}re systems respectively
in terms of the geometric structure and we show that these estimates are sharp
(Proposition~\ref{Hess=0} and Theorem~\ref{maximal symmetry}). }


\Classification{58K20; 53A15; 53C42}

\section{Introduction}

{\bf 1.1.}
The second-order partial dif\/ferential equation
\begin{gather*}
\mathrm{Hess}(f)=
\det\left(
\frac{\partial^2 f}{\partial x_i\partial x_j}
\right)_{1\leq i, j \leq n} = c \quad (\text{$c$ is constant}, \ c \not= 0),
\end{gather*}
for a scalar function $f$ of $n$ real variables $x_i$, $i = 1, 2, \dots, n$,
describes improper (parabolic)
af\/f\/ine hyperspheres
$
z = f(x_1, \dots, x_n)
$
and it
plays a signif\/icant role in equi-af\/f\/ine geometry (see \cite{N-S} for example).
Similarly the equation of constant Gaussian (Gauss--Kronecker) curvature
\begin{gather*}
K = c   \quad (\text{$c$ is constant})
\end{gather*}
for hypersurfaces is important
in Riemannian geometry (see \cite{K-N} for example).
Note that it is written, for graphs $z = f(x_1, \dots, x_n)$,
as the equation
\begin{gather*}
\mathrm{Hess}(f) =
(-1)^n c \big(1 + p_1^2 + \cdots + p_n^2\big)^{\frac{n+2}{2}},
\end{gather*}
where $p_i = \frac{\partial f}{\partial x_i}$.
Therefore the equations
${\mathrm{Hess}}(f) = c$ and
$K = c$ are regarded as Monge--Amp{\`e}re
equations, and they are studied from geometric aspects in this paper.
If we treat these equations in the framework of
{\it Monge--Amp{\`e}re systems},
then we realize that they have a specif\/ic character.

In~\cite{I-M}, we treated improper af\/f\/ine spheres and
constant Gaussian curvature surfaces in~${\mathbf{R}}^3$
from the view point of Monge--Amp{\`e}re equations of two variables,
and we analyzed the singula\-ri\-ties of their geometric solutions.
There we ef\/fectively used the direct sum decomposition of
the standard contact structure $D \subset T{\mathbf{R}}^5$ on~${\mathbf{R}}^5$
into a pair of two Lagrangian plane f\/ields~$E_1$,~$E_2$,
namely a {\it Lagrangian pair}.

Based on the notion of
Lagrangian pairs generalized to the higher-dimensional cases, namely, for
contact manifolds of dimension $2n+1$,
we introduce decomposable and bi-decomposable Monge--Amp{\`e}re systems with Lagrangian pairs
in Section~\ref{Monge-Ampere systems and Lagrangian pairs}.
A decomposable (resp.\ a~bi-decomposable) Monge--Amp{\`e}re system is def\/ined by
a decomposable $n$-form (resp.\ a~sum of two decomposable $n$-forms) which is compatible
with the underlying Lagrangian pair.
The class of Monge--Amp{\`e}re systems with Lagrangian pairs, which is introduced in this paper,
is invariant under contact transformations. If a Monge--Amp{\`e}re system is isomorphic to a Monge--Amp{\`e}re system
with a Lagrangian pair, then it is accompanied with a Lagrangian pair (see Def\/inition~\ref{isomorphism}).
On the other hand, it is not trivial that
the Lagrangian pair is uniquely asso\-cia\-ted to a given Monge--Amp{\`e}re
system. Then we are led to natural questions:
Is the Lagrangian pair uniquely determined
by the decomposable (resp.\ the bi-decomposable form) form?
Is the Lagrangian pair recovered
only from the Monge--Amp{\`e}re system?

We see that Lagrangian pair is not determined by the decomposable form.
Any decomposable Monge--Amp{\`e}re system with Lagrangian pair $(E_1, E_2)$
is of Lagrangian type in the sense of
\cite{M-M}, and the Lagrangian subbundle $E_1$ is obtained as the characteristic system of the Monge--Amp{\`e}re system (see also
\cite{Mo2} and \cite[Chapter~V]{B-C-3G}).
However the complementary Lagrangian subbundle~$E_2$ is not uniquely determined.

In Section~\ref{Lagrangian pair and bi-decomposable forms.},
we show a close relation between Lagrangian pairs
and bi-decomposable forms, and give an answer to the above
questions by showing that a bi-decomposable Monge--Amp{\`e}re system has the unique
$n$-form as a local generator
up to a multiplication by a non-zero function and modulo the contact form
(Theorem~\ref{bi-decomposable-form}) and that
such a bi-decomposable form uniquely determines the associated Lagrangian pair $(E_1, E_2)$ uniquely,
provided $n\geq 3$ (Theorem~\ref{Uniqueness}).
Thus we see that any automorphism of a given Monge--Amp{\`e}re system
with Lagrangian pair induces an automorphism of the underlying Lagrangian pair, if $n \geq 3$.
Note that Lagrangian pair is not determined by the bi-decomposable form
if $n = 2$ (Remark~\ref{n=2}).

{\bf 1.2.}
It follows that
the study of Monge--Amp{\`e}re systems with Lagrangian pairs has close relation with
the theory of Takeuchi \cite{Tak} on ``Lagrangian contact structures''.

From the viewpoint of geometric structures,
the comparison of the Lagrangian contact structures and Monge--Amp{\`e}re systems with Lagrangian pairs
goes as follows:
we treat Monge--Amp{\`e}re systems with
Lagrangian pairs on $M = P(T^*W)$, the projective cotangent bundle
over a manifold~$W$ of dimension $n+1$ in the following two cases.

For the f\/irst case, if the base space
$W$ has an af\/f\/ine structure, then $M = P(T^*W)$ has the natural
Lagrangian contact structure, i.e., a Lagrangian pair, see~\cite{Tak}.
Moreover a Monge--Amp{\`e}re system with the Lagrangian pair on~$M$ is
naturally induced,  if~$W$ has an equi-af\/f\/ine structure.
Here an {\it equi-affine structure} on~$W$ means that
$W$ is equipped with a torsion-free linear connection and a parallel volume
form on $W$, see~\cite{N-S}.
Furthermore if $W$ is the af\/f\/ine f\/lat
${{\mathbf{R}}}^{n+1}$ or the torus $T^{n+1}$, then we have
the generalization of the Hessian constant equation ${\mathrm{Hess}}(f) = c$.

For the second case,
we take a Lagrangian contact structure on
$M=P(T^*W)$ or on the unit tangent bundle $T_1W$ over~$W$
with the projective structure induced from a Riemannian metric on~$W$.
Recall that the projective structure is def\/ined as the equivalence class
of the Levi-Civita connection, under the projective equivalence on
torsion-free linear connections which is determined by the set of un-parametrized geodesics.
Moreover we consider a Monge--Amp{\`e}re system with Lagrangian pair on~$M$
induced from the volume of the Riemannian metric on~$W$.
Furthermore if~$W$ is a projectively f\/lat
Riemannian manifold, that is, one of the spaces
${{\mathbf{E}}}^{n+1}$, $S^{n+1}$, $H^{n+1}$ with constant curvature,
we obtain the generalization of the Gaussian curvature constant equation $K = c$
as an ``Euler--Lagrange''  Monge--Amp\`ere system (see
Section~\ref{Hesse representations.}
and Sections~\ref{Homogeneous M-A systems with Lagrangian pairs.}.2--\ref{Homogeneous M-A systems with Lagrangian pairs.}.5).

We summarize those subjects as the chart:
\begin{gather*}
\begin{array}{@{}c@{\,}c@{\,}c@{}}
W^{n+1} & &  M^{2n+1}=P(T^*W)
\vspace{0.2truecm}
\\
\left[
{\begin{array}{@{}l@{}}
{\mbox{\rm an equi-af\/f\/ine structure, }}\\
{\mbox{\rm the volume structure of a Riemannian metric}}
\end{array}
}
\right]
\vspace{0.2truecm}
& \longleftrightarrow &
{\mbox{\rm a M-A system with Lagrangian pair}}
\\
\downarrow & & \downarrow
\vspace{0.2truecm}
\\
\left[
{\begin{array}{@{}l@{}}
{\mbox{\rm an af\/f\/ine structure,}} \\
{\mbox{\rm a projective structure}}
\end{array}
}
\right]
\vspace{0.2truecm}
& \longleftrightarrow &
{\mbox{\rm a Lagrangian contact structure}}
\end{array}
\end{gather*}
Here the lower row indicates the underlying structures and
the upper row indicates the additional structures.

{\bf 1.3.}
In Section~\ref{Lagrangian contact structures.} we recall the theory of Takeuchi.
In Section~\ref{Automorphisms}, we study the sym\-met\-ries of a Monge--Amp{\`e}re system
with a Lagrangian pair. Using the results in this paper,
we show that the local automorphisms of
a Monge--Amp{\`e}re system with
a Lagrangian pair form a f\/inite-dimensional Lie pseudo-group,
provided $n\geq 3$.
We determine the maximal dimension of the automorphism pseudo-groups
of the Monge--Amp{\`e}re systems with f\/lat Lagrangian pairs
(Theorem~\ref{maximal symmetry}).

Based on those aspects, we characterize a class of Monge--Amp{\`e}re systems
which includes the equations ${\mathrm{Hess}}(f) = c$ and $K = c$, $c \not= 0$.
In fact,
the class of Monge--Amp{\`e}re equations of type
\begin{gather*}
\mathrm{Hess}(f) = F(x_1, \dots, x_n, f(x), p_1, \dots, p_n), \qquad F \not= 0,
\end{gather*}
is characterized and called the class of
{\it Hesse Monge--Amp{\`e}re systems} in Section~\ref{Hesse representations.}.
We observe that the class of Hesse Monge--Amp{\`e}re systems
is invariant under the contact transformations in the cases $n \geq 3$
(Proposition~\ref{contact invariance}). For instance,
the Legendre dual of the above equation is well def\/ined and given by
\begin{gather*}
\mathrm{Hess}(f) =
\frac{1}{F\left(p_1, \dots, p_n,
\sum\limits_{i=1}^n x_i p_i - f(x), x_1, \dots, x_n\right)}.
\end{gather*}
Note that in the case $n = 2$,
${\rm Hess}(f) = \pm 1$ is transformed to
the Laplace equation \mbox{$f_{x_1x_1} {+} f_{x_2x_2} {=} 0$} or to the wave equation
$f_{x_1x_1} - f_{x_2x_2} = 0$.
Therefore the class of Hesse Monge--Amp{\`e}re systems is not invariant
under the contact transformations in the case $n = 2$.

In the case $n = 3$, any Monge--Amp{\`e}re system of the
class is given by a non-degenerate three-form
which is decomposed uniquely
up to ordering by two decomposable forms (see \cite{B1,H,K-L-R,L-R-C} and
also Section~\ref{Lagrangian pair and bi-decomposable forms.}).
This fact and
its generalizations are the basic reasons
behind the above observation.

Further, we provide the unif\/ied picture of various subclasses
of Monge--Amp{\`e}re
equations with signif\/icant examples in arbitrary dimensions from
various geometric frameworks.

In Section~\ref{A method to construct systems with Lagrangian pairs.},
we introduce the general method to construct
Euler--Lagrange Monge--Amp{\`e}re system.
We apply the method of construction to several situations and obtain
several illustrative examples.
In Section~\ref{Homogeneous M-A systems with Lagrangian pairs.},
based on the general method, we show that
homogeneous Monge--Amp{\`e}re systems with
f\/lat Lagrangian pairs arise
in a very natural manner,
in equi-af\/f\/ine geometry, in Euclidean geometry,
in sphere geometry, in hyperbolic geometry, and moreover
in Minkowski geometry.
For these geometries, we construct
Monge--Amp{\`e}re systems with Lagrangian pairs explicitly and globally.
Moreover we show that the estimate proved in Section~\ref{Automorphisms} is
best possible by providing the example with the maximal symmetry.

{\bf 1.4.}
In this paper we treat, as underlying manifolds for Monge--Amp{\`e}re
systems, contact manifolds of dimensions $\geq 5$.
We remark that $3$-dimensional contact manifolds with
Lagrangian pairs
are related to second-order ordinary dif\/ferential equations
with normal form. They are studied in detail in \cite{A,I-L}.

As in \cite{B1,B2,K-L-R, L-R-C},
Lagrangian pairs can be formulated, at least locally,
on symplectic manifolds by means of reduction process.
In this paper we adopt
the contact framework of Monge--Amp{\`e}re equations based on Lagrangian
contact structures.

We remark that the paper \cite{M-M} treats in detail
the class of decomposable Monge--Amp{\`e}re systems
or Monge--Amp{\`e}re systems with one decomposability, which
are modeled on the Monge--Amp{\`e}re equation
$\mathrm{Hess}(z) =0$.

We will solve the equivalence problem of Monge--Amp{\`e}re systems with Lagrangian
pairs in the subsequent paper.

\section[Monge-Amp{\`e}re systems and Lagrangian pairs]{Monge--Amp{\`e}re systems and Lagrangian pairs}
\label{Monge-Ampere systems and Lagrangian pairs}

We start with the general def\/inition of Monge--Amp{\`e}re systems
\cite{B-G-G,L, Mo1,Mo2}.
Recall that a~contact structure $D$ on a~mani\-fold~$M$
is a subbundle of~$TM$ of codimension one locally def\/ined by a~contact
$1$-form $\theta$ by $D = \{ \theta = 0\}$ such that~$d\theta$ is non-degenerate,
that is, symplectic, on the bundle~$D$.
A~manifold endowed with a contact structure is called a contact manifold.
It is known that the dimension of a contact manifold is odd.

Let $(M, D)$ be a contact manifold of dimension $2n+1$ with
the contact structure \mbox{$D \subset TM$}.
A~{\it Monge--Amp{\`e}re system}
on $M$ is by def\/inition an exterior dif\/ferential system ${\mathcal M} \subset \Omega_M$
ge\-ne\-rated locally by a contact form $\theta$ for $D$ and an
$n$-form $\omega$ on~$M$: for each point $x \in M$,
there exists an open neighborhood~$U$ of~$x$ in~$M$ such that,
algebraically,
\begin{gather*}
{\mathcal M}\vert_U =
\langle \theta, d\theta, \omega
\rangle_{\Omega_U}.
\end{gather*}
Here $\Omega_M$ (resp.\ $\Omega_U$)
is the sheaf of germs of
exterior dif\/ferential forms
on $M$ (resp.\ on~$U$). In this case we call
$\omega$ a {\it local generator} of the Monge--Amp{\`e}re system
${\mathcal M}$ (modulo the contact ideal $\langle \theta, d\theta \rangle$).
Note that one may assume that $\omega$ is ef\/fective, i.e.,
$d\theta \wedge \omega \equiv
0\ (\rm{mod}\ \theta)$.
Also note that the $(n+1)$-form
$d\omega$ belongs to the contact ideal locally necessarily (see \cite{B1,B2, L}).

Let $D \subset TM$ be a vector bundle of rank $2n$ in the tangent bundle of a manifold $M$.
Recall that a conformal symplectic structure on $D$ is a reduction of the structure group of $D$
to the conformal symplectic group ${\rm CSp}({\mathbf{R}}^{2n})$.
If $(M, D)$ is a contact manifold of dimension $2n+1$, then
the conformal symplectic structure on $D$ is def\/ined locally by $d\theta$ for a contact form
$\theta$ which gives $D$ locally.
In particular, for each point $x \in M$, $D_x$ has the symplectic structure which is determined uniquely up to
a multiplication of a non-zero constant.
We call a linear subspace $W \subset D_x$
{\it Lagrangian} if $W$ is isotropic for the conformal symplectic structure on $D_x$ and $\dim_{{\mathbf{R}}} W = n$.
A subbundle $E \subset D$ is called a {\it Lagrangian subbundle} if $E_x \subset D_x$ is Lagrangian for any $x \in M$.

Now we def\/ine the key notion in this paper.

\begin{Def}
Let $(M, D)$ be a contact manifold.
A {\it Lagrangian pair} is a pair $(E_1, E_2)$
of Lagrangian subbundles of $D$ with respect to the conformal symplectic structure
on $D$ which satisf\/ies the condition $D = E_1 \oplus E_2$.
\end{Def}

\begin{Rem}
In \cite[\S~5.2]{B}, the notion of bi-Lagrangian structure is def\/ined as the transverse pair of Lagrangian foliations in a symplectic manifold.
Since we treat the contact case, it might be natural to use
the terminology  ``Legendrian'' instead of ``Lagrangian''.
However we would like to use ``Lagrangian'' in the general cases, and
to use the terminology ``Legendrian'' just for the integrable cases, such as ``Legendrian submanifolds'' and
``Legendrian f\/ibrations'' (see Section~\ref{Hesse representations.}).
\end{Rem}

The standard example of Lagrangian pair is given as follows.

{\bf The standard example.}
The standard example of Lagrangian pair is given on
the standard Darboux model.
Consider ${\mathbf{R}}^{2n+1}$ with coordinates
$(x_1,\dots,x_n,z,p_1,\dots,p_n)$ and with the standard
contact structure
$D_{\rm{st}} = \{ v \in TM \,|\, \theta_{\rm{st}}(v) = 0\}$ def\/ined by
the standard contact form
$\theta_{\rm{st}} = dz - \sum\limits_{i=1}^np_idx_i$.
Then we set
\begin{gather*}
E^{\rm{st}}_1   =
\{ v \in D_{\rm{st}} \,|\,
dp_1(v) = \cdots = dp_n(v) = 0 \} = \left\langle \frac{\partial}{\partial x_1} +
p_1\frac{\partial}{\partial z},
\dots, \frac{\partial}{\partial x_n} + p_n\frac{\partial}{\partial z}\right\rangle,
\\
E^{\rm{st}}_2  =
\{ u \in D_{\rm{st}} \,|\,
dx_1(u) = \cdots = dx_n(u) = 0\} =
\left\langle \frac{\partial}{\partial p_1},
\dots, \frac{\partial}{\partial p_n}\right\rangle.
\end{gather*}
Then $(E^{\rm{st}}_1, E^{\rm{st}}_2)$ is a Lagrangian pair on
$({\mathbf{R}}^{2n+1}, D_{\rm{st}})$.

Note that,
in \cite{Tak}, Takeuchi called
a contact structure $D$ endowed with a Lagrangian pair $(E_1,E_2)$
a {\it Lagrangian contact structure} $(D; E_1,E_2)$ and
gave a detailed study on this geometric structure.

Moreover we consider an exterior dif\/ferential system associated to a given
Lagrangian pair or Lagrangian contact structure.

\begin{Def}
A Monge--Amp{\`e}re system $\mathcal{M}$ is called
a {\it decomposable Monge--Amp{\`e}re system with a Lagrangian pair} $(E_1,E_2)$
if in a neighborhood of each point of $M$, there exists a local generator $\omega$
of $\mathcal{M}$
satisfying the following {\it decomposing condition}:
\begin{gather*}
i_u\omega = 0\ (u\in E_1), \quad \text{and} \quad
\omega\vert_{E_2}\ \text{is a volume form on}\ E_2.
\end{gather*}
\end{Def}

Such a decomposable Monge--Amp{\`e}re system
of Lagrangian type is a decomposable Monge--Amp{\`e}re system with
the characteristic system $E_1$ in the sense of~\cite{M-M}.
Note that this class of Monge--Amp{\`e}re equations have been introduced in~\cite{G}.
Also note that decomposable Monge--Amp{\`e}re systems are studied from another perspective in~\cite{A-B-M-P} by the name of Goursat equations.

\begin{Def}
A Monge--Amp{\`e}re system $\mathcal{M}$ is called
a {\it bi-decomposable Monge--Amp{\`e}re system with a Lagrangian pair} $(E_1,E_2)$
if, around each point of $M$, there exists a local generator~$\omega$
of~$\mathcal{M}$
of the form
\begin{gather*}
\omega=\omega_1-\omega_2
\end{gather*}
by $n$-forms $\omega_1$ and $\omega_2$
satisfying the following {\it bi-decomposing condition}:
\begin{gather*}
i_u\omega_1=0\quad (u\in E_2),\qquad i_v\omega_2=0\quad (v\in E_1),
\\
\omega_1| _{E_1}\ \text{is a volume form on}\ E_1,\quad
\mathrm{and}\quad \omega_2| _{E_2}\ \text{is a volume form on}\ E_2.
\end{gather*}
We call such an $n$-form $\omega$ a {\it bi-decomposable form} and
$(\omega_1, \omega_2)$ a {\it bi-decomposition} of~$\omega$.
Given a~Lagrangian pair $(E_1,E_2)$ on $(M, D)$ and $n$-forms $\omega_1$,
$\omega_2$ satisfying the bi-decomposing condition for $(E_1,E_2)$,
we def\/ine a Monge--Amp{\`e}re system
${\mathcal M}$ with Lagrangian pair
by setting $\omega = \omega_1-\omega_2$.
\end{Def}

\begin{Rem}
An immersion $f\colon  L \to M$ of an $n$-dimensional manifold $L$ to $M$
is called a~{\it geometric solution} of a Monge--Amp{\`e}re system
${\mathcal M} = \langle \theta, d\theta, \omega\rangle$ if
$f^*{\mathcal M} = 0$, namely, if $f^*\theta = 0$, i.e., $f$ is a Legendrian
immersion, and $f^*\omega = 0$.

Any geometric solution to
a decomposable (resp.\ bi-decomposable) Monge--Amp{\`e}re system~${\mathcal M}$
with a Lagrangian pair $D = E_1\oplus E_2$
on a contact manifold $(M^{2n+1}, D)$ has a crucial property.

In fact, an immersion $f\colon L \to M$ is a geometric solution of a decomposable Monge--Amp{\`e}re system
with a~Lagrangian pair $D = E_1\oplus E_2$ if and only if $f_*(T_pL) \cap (E_1)_{p} \not= \{ 0\}$.

For a bi-decomposable Monge--Amp{\`e}re system ${\mathcal M}$
with a Lagrangian pair $D = E_1\oplus E_2$,
suppose that $f\colon L \to M$ is a geometric solution of ${\mathcal M}$.
Then, for any $p \in L$, we see that
\begin{gather*}
E_1 \cap f_*(T_pL) = \{ 0\},\qquad \text{if and only if}\quad E_2 \cap f_*(T_pL) = \{ 0\},
\end{gather*}
equivalently,
\begin{gather*}
E_1 \cap f_*(T_pL) \not= \{ 0\},\qquad \text{if and only if}\quad E_2 \cap f_*(T_pL) \not= \{ 0\}.
\end{gather*}
In fact, let $W$ be an $n$-plane
in $D_p = (E_1)_p \oplus (E_2)_p$ satisfying
$\omega\vert_W = 0$.
The direct sum decomposition
def\/ines projections $\pi_1\colon D_p \to (E_1)_p$ and $\pi_2\colon D_p \to (E_2)_p$.
Then $E_1 \cap W = \{ 0\}$ if and only if $\pi_2\vert_{W}\colon W \to
(E_2)_p$ is an isomorphism, and
$E_2 \cap W = \{ 0\}$ if and only if $\pi_1\vert_{W}\colon W \to
(E_1)_p$ is an isomorphism, respectively.
For any bi-decomposition $\omega = \omega_1 - \omega_2$ and
for any basis $u_1, \dots, u_n$ of $W$, we have
\begin{gather*}
\omega_1(\pi_1(u_1), \dots, \pi_1(u_n)) =
\omega_1(u_1, \dots, u_n) =
\omega_2(u_1, \dots, u_n) =
\omega_2(\pi_2(u_1), \dots, \pi_2(u_n)).
\end{gather*}
Using the bi-decomposing condition again,
we see that $\pi_1\vert_{W}$ is an isomorphism if and only if
the most left hand side is non-zero, and it is equivalent to
the condition that
$\pi_2\vert_{W}$ is an isomorphism.
\end{Rem}

Now we are led to natural questions:
\begin{itemize}\itemsep=0pt
\item Is the Lagrangian pair $(E_1,E_2)$ uniquely determined
by a decomposable form?

\item Is the Lagrangian pair $(E_1,E_2)$ uniquely determined
by a bi-decomposable form?

\item Is the Lagrangian pair $(E_1, E_2)$ recovered
only from the Monge--Amp{\`e}re system $\mathcal{M}$?
\end{itemize}

As is stated in Introduction, the f\/irst question is answered negatively.
To answer the second and third questions, we recall the basic def\/initions.

\begin{Def}
\label{isomorphism}
Let $(M, D)$,  $(M', D')$ be contact manifolds of dimension $2n+1$,
and ${\mathcal M}$, ${\mathcal M'}$ Monge--Amp{\`e}re systems
on contact manifolds $(M,D)$, $(M',D')$ respectively.
A dif\/feomorphism $\Phi\colon M \longrightarrow M'$ is called
an {\it isomorphism of Monge--Amp{\`e}re systems}
if (1)~$\Phi$ is a~contactomorphism,
namely $(\Phi_*)D=D'$, and
(2)~${\Phi}^*{\mathcal M'}={\mathcal M}$.
\end{Def}

Now suppose that contact manifolds $(M, D)$,  $(M', D')$ of dimension
$2n+1$ are endowed with
Lagrangian pairs $(E_1, E_2)$, $(E'_1, E'_2)$ respectively,
namely, that the decompositions $D = E_1\oplus E_2$ and $D' = E'_1\oplus E'_2$
are given. Suppose $n \geq 3$.
Then, from the result in
Section~\ref{Lagrangian pair and bi-decomposable forms.}
mentioned above, we have that
any isomorphism ${\Phi}$ of a Monge--Amp{\`e}re system
${\mathcal M}$ with the Lagrangian pair $(E_1, E_2)$ and
a Monge--Amp{\`e}re system
${\mathcal M}'$ with the Lagrangian pair $(E'_1, E'_2)$
necessarily preserves the Lagrangian pairs up to ordering, namely,
$(\Phi_*)E_1 = E'_1$, $(\Phi_*)E_2 = E'_2$  or
$(\Phi_*)E_1 = E'_2$, $(\Phi_*)E_2 = E'_1$.

In the case $n = 2$, a Lagrangian pair is not uniquely determined
from a Monge--Amp{\`e}re system and moreover, the automorphism pseudo-group
may be of inf\/inite dimension. For instance, consider the equation
${\rm Hess} = -1$ which is isomorphic to the linear wave
equation $f_{x_1x_1} - f_{x_2x_2} = 0$ and to the equation
$f_{x_1x_2} = 0$. Then the last equation has inf\/inite-dimensional
automorphisms induced by dif\/feomorphisms $(x_1, x_2)
\mapsto (X_1(x_1), X_2(x_2))$.

\section{Lagrangian pair and bi-decomposable form}
\label{Lagrangian pair and bi-decomposable forms.}

Let $(M, D)$ be a contact manifold of dimension $2n+1$ with
a contact structure $D \subset TM$.
We have def\/ined in Section~\ref{Monge-Ampere systems and Lagrangian pairs} the notion of bi-decomposing
conditions and bi-decomposable forms on~$(M, D)$.
Then, f\/irst we show

\begin{Lem}
\label{bi-decomposability}
Let $\omega$ be an $n$-form on a contact manifold $(M, D)$,
$D = \{ u \in TM \,|\, \theta(u) = 0\}$
for a local contact form $\theta$ defining $D$,
and $(E_1, E_2)$
a Lagrangian pair of~$D$.
Assume that $\omega$ is
a~bi-decomposable form for $(E_1, E_2)$,
and $\omega = \omega_1 - \omega_2$ is
any bi-decomposition of $\omega$ for $(E_1, E_2)$.
Then locally there exists a coframe
$\theta, \alpha_1, \dots, \alpha_n, \beta_1, \dots, \beta_n$
of $T^*M$ such that
\begin{gather*}
E_1 = \{ v \in D \,|\, \beta_1(v) = \cdots = \beta_n(v) = 0\}, \qquad
E_2 = \{ u \in D \,|\, \alpha_1(u) = \cdots = \alpha_n(u) = 0\},
\end{gather*}
and that the $n$-forms
\begin{gather*}
\widetilde{\omega}_1 = \alpha_1\wedge \cdots \wedge \alpha_n,
\qquad
\widetilde{\omega}_2 = \beta_1\wedge \cdots \wedge \beta_n
\end{gather*}
satisfy the bi-decomposing condition for
$(E_1, E_2)$ with
\begin{gather*}
\widetilde{\omega}_1 \equiv \omega_1, \qquad
\widetilde{\omega}_2 \equiv \omega_2, \qquad
\omega \equiv \widetilde{\omega}_1 - \widetilde{\omega}_2
\end{gather*}
up to a multiple of $\theta$.
\end{Lem}

\begin{proof}
The proof is based on the fact that the symplectic group on a f\/inite-dimensional symplectic vector space
acts transitively on the set of transversal pairs of Lagrangian subspaces.

Let $X_1, \dots, X_n$ and $P_1, \dots, P_n$ be local frames
of~$E_1$ and~$E_2$ respectively.
Let~$R$ be the Reeb vector f\/ield
for a local contact form $\theta$ def\/ining $D$. Recall that~$R$ is def\/ined uniquely by $i_R \theta = 1$, $i_R d\theta = 0$.
Consider
the dual coframe $\theta, \alpha_1,\dots,\alpha_n, \beta_1,\dots,\beta_n$
of~$T^*M$
to the frame~$R$, $X_1, \dots, X_n$, $P_1, \dots, P_n$ of~$TM$.
Then we see, from the bi-decomposing condition,
that there exist an $(n-1)$-form $\gamma$ and a~non-vanishing function $\mu$ on $M$ such that
$\omega_1 = \mu(\alpha_1 \wedge \cdots \wedge \alpha_n) + \theta\wedge\gamma$.
By replacing $\alpha_1$ by $\frac1{\mu}\alpha_1$,
we may suppose $\mu \equiv 1$.
Similarly, we have $\omega_2 \equiv
\beta_1 \wedge \cdots \wedge \beta_n \ \rm{mod}\ \theta$.
\end{proof}

Next we show that $\mathcal{M}$ has the unique
local bi-decomposable generator $\omega$ modulo $\theta$.

\begin{Th}
\label{bi-decomposable-form}
Let $(E_1, E_2)$ be a Lagrangian pair and
$\omega$, $\omega'$ be two bi-decomposable $n$-forms
for the Lagrangian pair $(E_1, E_2)$ on $(M, D)$.
Assume that they generate the same Monge--Amp{\`e}re system
\begin{gather*}
{\mathcal M} =
\langle \theta, d\theta, \omega \rangle
=
\langle \theta, d\theta, \omega' \rangle.
\end{gather*}
Then there exist
locally a non-vanishing function $\mu$ and an $(n-1)$-form $\eta$
on $M$ such that $\omega' = \mu \omega + \theta \wedge \eta$.
\end{Th}

To show Theorem~\ref{bi-decomposable-form},
we study, for each $x \in M$, the symplectic exterior linear algebra
on the symplectic vector space $V = D_x$ with the symplectic form
$\Theta = d\theta\vert_{D_x}$ and with the decomposition
$V = V_1\oplus V_2$, $V_1 = (E_1)_x$, $V_2 = (E_2)_x$,
of $(V, \Theta)$ into Lagrangian subspaces.

Let $(V,\Theta)$ be a $2n$-dimensional symplectic vector space.
We say that an $n$-form $\omega \in \wedge^n V^*$ is
{\it bi-decomposable} if
there exist a decomposition $V = V_1\oplus V_2$ of~$V$ into
Lagrangian sub\-spa\-ces~$V_1$,~$V_2$ in~$V$ and
$n$-forms $\omega_1, \omega_2 \in \wedge^n V^*$ such
that $\omega = \omega_1 - \omega_2$,
$i_u\omega_1 = 0$ $(u \in V_2)$, $i_v\omega_2 = 0$ $(v \in V_1)$,
$\omega_1\vert_{V_1} \not= 0$, $\omega_2\vert_{V_2} \not= 0$.
In this case $(\omega_1, \omega_2)$ is called a {\it bi-decomposition}
of~$\omega$.
Then, similarly as the proof of Lemma~\ref{bi-decomposability}, we have that
there exists a basis
$\{ a_1,\dots, a_n, b_1,\dots, b_n\}$ of~$V$
such that
\begin{gather*}
\omega_1 = a_1^*\wedge \cdots \wedge a_n^*, \qquad
\omega_2 = b_1^*\wedge \cdots \wedge b_n^*,
\end{gather*}
and that
\begin{gather*}
V_1= \langle a_1,\dots, a_n\rangle,
\qquad
V_2= \langle b_1,\dots, b_n\rangle,
\end{gather*}
where $\{ a_1^*,\dots, a_n^*, b_1^*,\dots, b_n^*\}$ denotes the dual basis
of $\{ a_1,\dots, a_n, b_1,\dots, b_n\}$.
Note that
$V_2$ (resp.~$V_1$) coincides with the annihilator of $a_1^*,\dots, a_n^*$
(resp.~$b_1^*,\dots, b_n^*$).
Moreover we see that
there exist a symplectic basis $\{ a_1,\dots, a_n, b_1,\dots, b_n\}$
of $(V, \Theta)$
and a non-zero constants~$a$,~$b$ such that
\begin{gather*}
\omega =
a   a_1^*\wedge \cdots \wedge a_n^* - b   b_1^*\wedge \cdots \wedge b_n^*.
\end{gather*}
If we replace $a_1$, $b_1$ by $(1/b) a_1$, $b b_1$
respectively,
so $a_1^*$, $b_1^*$ by $b a_1^*$, $(1/b) b_1^*$,
and set $c = ab$, then we have the following:

\begin{Lem}
\label{symplectic basis}
Let $\omega \in \wedge^n V^*$ be a bi-decomposable $n$-form on
a $2n$-dimensional symplectic vector space $(V, \Theta)$, and
$\omega = \omega_1 - \omega_2$ a~bi-decomposition of~$\omega$.
Then there exist a symplectic basis
$\{ a_1,\dots, a_n, b_1,\dots, b_n\}$ of~$(V, \Theta)$
and a non-zero constant $c$ such that
\begin{gather*}
\omega_1  =
c  a_1^*\wedge \cdots \wedge a_n^*, \qquad
\omega_2  =   b_1^*\wedge \cdots \wedge b_n^*,
\\
\omega  =  c  a_1^*\wedge \cdots \wedge a_n^* -
b_1^*\wedge \cdots \wedge b_n^*,
\qquad
\Theta  =  a_1^* \wedge b_1^* + \cdots + a_n^* \wedge b_n^*.
\end{gather*}
\end{Lem}

A form $\varphi \in \wedge^n V^*$ on
a symplectic vector space $(V,\Theta)$
is called {\it effective} if
the interior product $i_{X_{\Theta}}\varphi = 0$ for
the $2$-vector $i_{X_{\Theta}}$ corresponding to the
$2$-form~$\Theta$.
Note that
\begin{gather*}
X_{\Theta} = a_1 \wedge b_1 + \cdots + a_n \wedge b_n
\end{gather*}
in terms of the basis in Lemma \ref{symplectic basis}.
That $\varphi \in \wedge^n V^*$ is ef\/fective if and only if
the wedge $\varphi\wedge\Theta$ with the symplectic form~$\Theta$ is equal to zero~\cite{B2, L}.
Then we have, by Lemma~\ref{symplectic basis}:

\begin{Cor}
\label{effective}
If $\omega \in \wedge^n V^*$ is bi-decomposable and
$\omega = \omega_1 - \omega_2$ is any bi-decomposition, then~$\omega$,~$\omega_1$ and~$\omega_2$ are all effective.
\end{Cor}

We will apply the following basic result to our situation.

\begin{Lem}[\protect{\cite[Theorem~1.6]{L}}]\label{Lych}
Let $\omega$, $\omega'$ be effective $k$-forms on a symplectic vector space
$(V, \Theta)$, $(0\leq k\leq n)$.
Suppose that, for every isotropic subspace $L\subset V$ on which
$\omega\vert_L = 0$, the form $\omega'$ also vanishes on~$L$,
$\omega'\vert_L = 0$.
Then we have $\omega' =\mu{\omega}$ for some~$\mu \in {\mathbf{R}}$.
\end{Lem}

Then we have the following:

\begin{Lem}
\label{bi-decomposable2}
Let $\omega$, ${\omega}'$ be bi-decomposable forms on $(V, \Theta)$.
Suppose that~${\omega}'$ is of the form
$\lambda \omega+\phi \wedge \Theta$
for a scalar $\lambda$ and an $(n-2)$-form $\phi$.
Then ${\omega}' = \mu\omega$ for a scalar~$\mu$.
\end{Lem}

\begin{proof}
By Corollary \ref{effective}, the $n$-form $\omega$ is ef\/fective
and so is the $n$-form~${\omega}'$.
For every Lagrangian subspace~$L$ $(\Theta\vert_L = 0)$ on which
$\omega\vert_L = 0$,
the form $\omega'$ also vanishes.
Therefore, by Lemma~\ref{Lych},
it follows that
$\omega' =\mu{\omega}$
for a scalar~$\mu$.
\end{proof}

\begin{Rem}
Lemma \ref{bi-decomposable2} can be shown, as an alternative proof, by
applying Lefschetz isomorphism $\Theta^2\colon
\wedge^{n-2} V \to \wedge^{n+2} V$. In fact, from ${\omega}' - \mu\omega = \phi \wedge \Theta$
we have $\phi\wedge \Theta^2 = 0$, which implies~$\phi = 0$.
\end{Rem}

\begin{proof}[Proof of Theorem \ref{bi-decomposable-form}]
Since $\omega'$ belongs to ${\mathcal M} = \langle \theta, d\theta, \omega
\rangle$,
we set $\omega' = \lambda\omega +  d\theta\wedge\phi + \theta\wedge \eta$,
for a~function~$\lambda$, an $(n-2)$-form $\phi$ and an $(n-1)$-form $\eta$
on~$M$.
For each $x \in M$, we have
$\omega'\vert_{D_x} = \lambda(x)\omega\vert_{D_x} +
\Theta\wedge\phi\vert_{D_x}$,
where $\Theta = d\theta\vert_{D_x}$.
Then, by Lemma~\ref{bi-decomposable2}, we have
$\omega'\vert_{D_x} = \mu(x)\omega\vert_{D_x}$ for a scalar~$\mu(x)$
depending on $x \in M$.
Since $\omega\vert_{(E_1)_x}$ is a volume form, we see that~$\mu(x)$
is unique and
of class $C^\infty$. Moreover there exists a $C^\infty$ $(n-1)$-form $\eta$
such that $\omega' - \mu\omega = \theta \wedge \eta$, which implies
the required consequence.
\end{proof}

Moreover we show the following basic result:
\begin{Th}\label{Uniqueness}
Assume that $n\geq 3$.
Let
$\mathcal{M}=
\langle \theta, d\theta, \omega \rangle$ be
a Monge--Amp{\`e}re system
with a~Lag\-rangian pair locally generated by a bi-decomposable
$n$-form~$\omega$. Then~$\mathcal{M}$
canonically determines the Lagrangian pair~$(E_1,E_2)$.
Namely, if $(\omega_1, \omega_2)$
$($resp.\ $(\omega'_1, \omega'_2))$
is a bi-decomposition of~$\omega$ for
a~Lag\-rangian pair $(E_1,E_2)$ $($resp.~$(E_1', E_2'))$
enjoying the bi-decomposing condition.
Then we have
\begin{gather*}
E_1'=E_1,\quad  E_2'=E_2, \qquad {\mathrm or} \qquad
E_1' = E_2, \quad  E_2' = E_1.
\end{gather*}
\end{Th}

Theorem \ref{Uniqueness} follows immediately from the following:

\begin{Prop}
\label{uniqueness 2}
Assume that $n\geq 3$.
Let $(V,\Theta)$ be a $2n$-dimensional symplectic vector space.
Let $\omega=\omega_1-\omega_2$ be a bi-decomposable
$n$-form for a Lagrangian pair $(V_1, V_2)$ of $V$.
Then the bi-decomposition $(\omega_1, \omega_2)$ of~$\omega$
is unique up to ordering:
If $\omega=\omega'_1-\omega'_2$ is another bi-decomposition
of $\omega$ for another Lagrangian pair $(V'_1, V'_2)$ of~$V$, then
\begin{gather*}
V'_1 = V_1, \qquad V'_2 = V_2, \qquad \omega'_1 = \omega_1, \qquad \omega'_2 = \omega_2,
\qquad {\mathrm{or}} \\
V'_1 = V_2, \qquad V'_2 = V_1, \qquad \omega'_1 = - \omega_2, \qquad \omega'_2 = - \omega_1.
\end{gather*}
\end{Prop}

The bi-decomposition of $\omega$ (Proposition~\ref{uniqueness 2})
is given by using the symplectic structure,
based on Hitchin's result \cite[Propositions 2.1,~2.2]{H},
in the case that $n$ is odd and $n \geq 3$, as follows:

Let $\omega = \omega_1 - \omega_2$, with
$\omega_1 = c a_1^*\wedge \cdots \wedge a_n^*$
and
$\omega_2 = b_1^*\wedge \cdots \wedge b_n^*$, as in
Lemma~\ref{symplectic basis}.
Let $\epsilon=a^*_1\wedge \cdots \wedge a^*_n
\wedge b^*_1\wedge \cdots b^*_n$
be the associated basis vector for $\Lambda^{2n}V^*$,
which is the
intrinsically def\/ined volume form by the symplectic structure on~$V$.
From the isomorphism $A\colon \Lambda^{2n-1}V^* \longrightarrow
V\otimes \wedge^{2n}V^*$ induced by the exterior product
$V^*\otimes\wedge^{2n-1}V^* \to \wedge^{2n}V^*$,
we def\/ine, for each $\psi \in \wedge^nV^*$,
a linear transformation
$K_{\psi}=K_{\psi}^{\epsilon}\colon V \longrightarrow V$ by
$K_{\psi}(u)\epsilon = A(i_u(\psi)\wedge \psi)$, and put
$\lambda(\psi)=\lambda_{\epsilon}(\psi)=
\frac{1}{2n}\operatorname{tr}(K_{\psi}^2)$.

In particular, we set $\psi = \omega$.
Then we have $K_{\omega}a_i=-c a_i$, $K_{\omega}b_i
= (-1)^{n+1} c b_i = cb_i$.
Then we have $K_{\omega}^2 = c^2\operatorname{id}$,
therefore $\lambda(\omega) = c^2$.
Since $K_{\omega}^*a_i^* = -ca_i^*$, $K_{\omega}^*b_i^* =cb_i^* $,
we have
\begin{gather*}
K_{\omega}^*\omega = c(-c)^n a^*_1\wedge \cdots \wedge a^*_n -
c^{n}b^*_1\wedge \cdots \wedge b^*_n =
(-c)^n(\omega_1 + \omega_2).
\end{gather*}
Thus, from
$\omega={\omega}_1-{\omega}_2$, $-\frac{1}{c^n}K_{\omega}^*\omega={\omega}_1+{\omega}_2$
and $\lambda(\omega) = c^2$ $(c >0, <0)$,
we see
\begin{gather*}
\omega_1 = \frac{1}{2}\big(\omega \mp \lambda(\omega)^{-\frac{n}{2}}
K_{\omega}^*\omega\big),\qquad
\omega_2= \frac{1}{2}\big({-}\omega \mp \lambda(\omega)^{-\frac{n}{2}}
K_{\omega}^*\omega\big).
\end{gather*}
Since $K_{\omega}^*\omega$ and $\lambda(\omega)$ are
intrinsically determined from the symplectic structure on~$V$,
so is the decomposition of~$\omega$.

To verify Proposition~\ref{uniqueness 2} in general case,
we observe a fact from projective geometry of Pl\"{u}cker embeddings.
Consider Grassmannian ${\rm Gr}(n, V^*)$ consisting of all
$n$-dimensional subspaces
in the $2n$-dimensional vector space $V^*$, and its
Pl\"{u}cker embedding ${\rm Gr}(n, V^*) \hookrightarrow P(\wedge^{n} V^*)$.

\begin{Lem}
\label{proj.geom}
Let $\Lambda_1, \Lambda_2 \in {\rm Gr}(n, V^*)$ with $\Lambda_1 \cap \Lambda_2 = \{ 0\}$.
\begin{enumerate}\itemsep=0pt
\item[$(1)$] The projective line in $P(\wedge^{n} V^*)$
through the two points~$\Lambda_1$,~$\Lambda_2$ does not intersect
with ${\rm Gr}(n, V^*)$ at other points.

\item[$(2)$] Assume $n \geq 3$. Let $\Lambda_3$ be another point in~${\rm Gr}(n, V^*)$
different from $\Lambda_1$, $\Lambda_2$. Then the projective plane
in $P(\wedge^{n} V^*)$ through the three points $\Lambda_1$, $\Lambda_2$,
$\Lambda_3$
does not intersect with ${\rm Gr}(n, V^*)$ at other points.
\end{enumerate}
\end{Lem}

\begin{proof}
Choose a basis $e_1, \dots, e_n$ of $\Lambda_1$ and
$e_{n+1}, \dots, e_{2n}$ of $\Lambda_2$ to form
a basis of~$V^*$. Let $( p_{i_1, i_2, \dots, i_n})$
denote Pl\"{u}cker coordinates of the Pl\"{u}cker embedding. Here
$i_1, i_2, \dots, i_n$ are $n$-tuple chosen from $\{ 1, 2, \dots, 2n\}$.
For any sequences $1 \leq j_1 < j_2 < \cdots < j_{n-1} \leq 2n$ of
$(n-1)$-letters and
$1 \leq k_1 < k_2 < \cdots < k_{n+1} \leq 2n$ of $(n+1)$-letters,
we have the Pl\"{u}cker relation
\begin{gather*}
\sum_{\ell = 1}^{n+1} (-1)^\ell   p_{j_1, j_2, \dots, j_{n-1}, k_{\ell}}
p_{k_1, k_2,
\dots, \breve{k}_\ell, \dots, k_{n+1}} = 0
\end{gather*}
(see \cite{G-K-Z,G-H} for instance).
To see (1),  take a point $W  \in {\rm Gr}(n, V^*)$
on the projective line through~$\Lambda_1$,~$\Lambda_2$.
Since Pl\"{u}cker coordinates of $\Lambda_1$ (resp.\ $\Lambda_2$)
are given by $p_{1,\dots,n} = 1$ (resp.\ $p_{n+1, \dots, 2n} = 1$)
with zero other coordinates, we have that Pl\"{u}cker coordinates of~$W$
are given by $p_{1,\dots,n} = \lambda$, $p_{n+1, \dots, 2n} = \mu$
for some~$\lambda$,~$\mu$, other coordinates $p_{i_1, \dots, i_n}$ being
null.
Then, by applying the Pl\"{u}cker relation for $(j_1, \dots, j_{n-1}) =
(1, 2, \dots, n-1)$ and $(k_1, k_2, \dots, k_{n+1}) = (n-1, n, \dots, 2n)$,
we have $\lambda\mu = 0$. Therefore $\lambda = 0$ or $\mu = 0$, namely,
$W = \Lambda_1$ or $W = \Lambda_2$.

To see (2), let $\bar{p}=(\bar{p}_{i_1, i_2, \dots, i_n})$
be coordinates of~$\Lambda_3$.
Then the coordinates $p=(p_{i_1, i_2, \dots, i_n})$
for a point~$W$ on the plane
through $\Lambda_1$, $\Lambda_2$, $\Lambda_3$ are given by
\begin{gather*}
p_{1,\dots,n} = \lambda + \bar{p}_{1,\dots,n}, \qquad
p_{n+1, \dots, 2n} = \mu + \bar{p}_{n+1, \dots, 2n}, \qquad
p_{i_1, \dots, i_n} = \bar{p}_{i_1, \dots, i_n},
\end{gather*}
$\{i_1, \dots, i_n\} \not= \{ 1, \dots, n\}, \{ n+1, \dots, 2n\}$.
Suppose $W \in {\rm Gr}(n, V^*)$, $W \not= \Lambda_1, \Lambda_2, \Lambda_3$.
Then we see that
both~$p$ and~$\bar{p}$ satisfy
the Pl\"{u}cker relations and that $\lambda\mu\not= 0$.
Let $\{ i_1, \dots, i_n \} \not= \{ 1, \dots, n\}, \{ n+1, \dots, 2n\}$.
Then, because $n \geq 3$,  there exist $\ell$, $\ell'$, $\ell \not= \ell'$,
such that $i_\ell \not\in \{ n+1, \dots, 2n\}$ and $i_{\ell'}
\not\in \{ n+1, \dots, 2n\}$ or that $i_\ell \not\in \{ 1, \dots, n\}$
and $i_{\ell'} \not\in \{ 1, \dots, n\}$. In the former case,
we take $(i_1, \dots, \breve{i}_k, \dots, i_n)$,  $(i_k, n+1, \dots, 2n)$,
and write the Pl\"{u}cker relation to get $\mu\cdot\bar{p}_{i_1, \dots, i_n}
= 0$.
In the latter case, we take $(i_1, \dots, \breve{i}_k, \dots, i_n),
(1, \dots, n, i_k)$ to get $\lambda\cdot\bar{p}_{i_1, \dots, i_n} = 0$.
Therefore we obtain that $\bar{p}_{i_1, \dots, i_n} = 0$
for $\{ i_1, \dots, i_n \} \not= \{ 1, \dots, n\}, \{ n+1, \dots, 2n\}$.
This means that $\Lambda_3$
lies on the projective line through~$\Lambda_1$,~$\Lambda_2$,
which contradicts~(1).
\end{proof}

\begin{proof}[Proof of Proposition \ref{uniqueness 2}]
Set $\omega_1 - \omega_2 = \omega'_1 - \omega'_2$. Then
$\omega_1 - \omega_2 - \omega'_1$ is equal to a decomposable form
$- \omega'_2$.
Let $\Lambda_1$, $\Lambda_2$, $\Lambda'_1$, $\Lambda'_2$ be
points in ${\rm Gr}(n, V^*) \subset P(\wedge^n V^*)$
corresponding to $\omega_1$, $\omega_2$, $\omega'_1$, $\omega'_2$ respectively.
Then the projective plane through $\Lambda_1$, $\Lambda_2$, $\Lambda'_1$
intersects with~${\rm Gr}(n, V^*)$ also at~$\Lambda'_2$.
Note that $\Lambda'_1 \not= \Lambda'_2$ as well as
$\Lambda_1 \not= \Lambda_2$.
Suppose $\Lambda'_1 \not= \Lambda_1, \Lambda_2$.
By Lemma~\ref{proj.geom}(2), we have
$\Lambda'_2 = \Lambda_1$ or $\Lambda'_2 = \Lambda_2$.
Then we have that~$\Lambda'_1$ lies on the line through $\Lambda_1$, $\Lambda_2$.
By Lemma~\ref{proj.geom}(1), we conclude that
$\Lambda'_1 = \Lambda_1$ or $\Lambda'_1 = \Lambda_2$.
If $\omega'_1 = \lambda\omega_1$ for some $\lambda \not= 0$,
then $V'_2 = V_2$ as the annihilator of $\omega_1$ or $\omega'_1$.
If $\omega'_1 = \mu\omega_2$ for some $\mu \not= 0$,
then we have $V'_2 = V_1$.
By the symmetric argument, we obtain also that $V'_1 = V_1$ or $V'_1 = V_2$.
Thus we have $V_1' = V_1$, $V_2' = V_2$ or
$V_1' = V_2$,  $V_2' = V_1$.
Assume $V_1' = V_1$, $V_2' = V_2$, then,
restricting $\omega$ to $V = V_1 \oplus V_2$, we have
$\omega'_1 = \omega_1$, $\omega'_2 = \omega_2$.
Assume $V_1' = V_2$,  $V_2' = V_1$,
then we have $\omega'_1 = -\omega_2$, $\omega'_2 = - \omega_1$.
\end{proof}

\begin{Rem}
\label{n=2}
If $n=2$, that is, $M$ is of dimension~$5$,
then Theorem~\ref{Uniqueness} does not hold.
In fact, consider~$M={\mathbf{R}}^5$ with coordinates $(x,y,z,p,q)$
and with the contact form $\theta=dz-pdx-qdy$.
Take a $2$-form
\begin{gather*}
\omega=dx\wedge dy-dp\wedge dq.
\end{gather*}
Then decomposable $2$-forms $\omega_1 = dx\wedge dy$, $\omega_2=dp\wedge dq$
satisfy $\omega = \omega_1 - \omega_2$ and the bi-decomposing condition
for the Lagrangian pair given by
\begin{gather*}
E_1= \big\{ v \in T{\mathbf{R}}^5 \,|\, \theta(v) = dp(v) = dq(v) = 0 \big\}
= \left\langle \frac{\partial}{\partial x}+
p\frac{\partial}{\partial z},\frac{\partial}{\partial y}+
q\frac{\partial}{\partial z}\right\rangle,
\\
E_2= \big\{ u \in T{\mathbf{R}}^5 \,|\,
\theta(u) = dx(u) = dy(u) = 0\big\} = \left\langle \frac{\partial}{\partial p},
\frac{\partial}{\partial q}\right\rangle.
\end{gather*}
Then we can f\/ind other decomposable $2$-forms $\omega'_1=
d(x+p)\wedge dy$, $\omega'_2=dp\wedge d(y+q)$
satisfying $\omega=\omega'_1-\omega'_2$
and the bi-decomposing condition
for another Lagrangian pair given by
\begin{gather*}
E'_1  = \big\{ v \in T{\mathbf{R}}^5 \,|\, \theta(v) = dp(v) = d(y+q)(v) = 0\big\} =
\left\langle \frac{\partial}{\partial x}+p\frac{\partial}{\partial z},
\frac{\partial}{\partial y}+q\frac{\partial}{\partial z}-
\frac{\partial}{\partial q}\right\rangle,\\
E'_2  = \big\{ u \in T{\mathbf{R}}^5 \,|\, \theta(u) = d(x+p)(u) = dy(u) = 0 \big\}
= \left\langle \frac{\partial}{\partial x}+p\frac{\partial}{\partial z}-
\frac{\partial}{\partial p},\frac{\partial}{\partial q}\right\rangle.
\end{gather*}

Lemma~\ref{proj.geom}(2) does not hold in the case $n = 2$,
because ${\rm Gr}(2, {\mathbf{R}}^4) \hookrightarrow P(\wedge^2({\mathbf{R}}^4)) = P^5$
is a~hypersurface and a~projective plane intersects with~${\rm Gr}(2, {\mathbf{R}}^4)$
in inf\/inite points (a~planer curve).
\end{Rem}

From the above-mentioned propositions,
it follows that, in the case $n \geq 3$,
the notion of a bi-decomposable Monge--Amp{\`e}re system with a Lagrangian pair $(E_1,E_2)$
is nothing but
the notion of
a Monge--Amp{\`e}re system generated by a bi-decomposable
form $\omega=\omega_1-\omega_2$.

\section{Lagrangian contact structures}
\label{Lagrangian contact structures.}

Let us recall Takeuchi's paper \cite{Tak} for Lagrangian contact structures.
A contact structure with a Lagrangian pair is called a~Lagrangian contact structures in~\cite{Tak}.
A typical example of Lagrangian contact structures
is the projective cotangent vector bundle
$M = P(T^*W)$ of an $(n+1)$-dimensional manifold~$W$
with an af\/f\/ine structure (a~torsion-free linear connection) or
a~projective structure
(a~projective equivalence class of torsion-free linear connections).
For the canonical contact structure~$D$ on~$M$,
we take as a Lagrangian pair horizontal and vertical vector bundles
(cf.\ Example~\ref{hor+ver}).
In~\cite{Tak}, it is given the description of Cartan connections on~$P(T^*W)$
associated to Lagrangian contact structures and the equivalence of the vanishing
of the curvature of Cartan connection and the projective f\/latness of $W$.

The f\/lat model, which is a homogeneous space qualif\/ied as a model
for a Cartan connection, with Lagrangian contact structure is
the projective cotangent bundle $P(T^*P^{n+1})$
of the ($n+1$)-dimensional projective space $P^{n+1}=P(\mathbf{R}^{n+2})$.

Put $G={\rm PGL}(n+2,{\mathbf{R}})={\rm GL}(n+2,{\mathbf{R}})/C$ ($C$ is the center,
$C=\mathbf{R}^{\times}\cdot I_{n+2}$),
and $\mathfrak{g}=\operatorname{Lie} G \cong \mathfrak{sl}(n+2,{\mathbf{R}})$.
The Lie algebra $\mathfrak{g}$ has a structure of
a simple graded Lie algebra (GLA) of second kind as follows:
\begin{gather*}
       {\mathfrak g} = {\mathfrak{sl}}(n+2,{\mathbf{R}})  = {\mathfrak g}_{-2}\oplus {\mathfrak g}_{-1} \oplus
                    {\mathfrak g}_0
                    \oplus {\mathfrak g}_1 \oplus {\mathfrak g}_2 \\
\hphantom{{\mathfrak g}}{}
= \left\{{\begin{pmatrix}
                       0 & 0 & 0 \\
                       0 & \mathrm{O}_n & 0 \\
                       a & 0 & 0
                    \end{pmatrix}}\right\}
                       \oplus
                    \left\{{\begin{pmatrix}
                       0 & 0 & 0 \\
                       b_1 & \mathrm{O}_n & 0 \\
                       0 & {}^tb_2 & 0
                    \end{pmatrix}}\right\}
                       \oplus
                    \left\{{\begin{pmatrix}
                       \alpha & 0 & 0 \\
                       0 & A & 0 \\
                       0 & 0 &\beta
                    \end{pmatrix}}\right\} \\
\hphantom{{\mathfrak g}=}{} \quad {}\oplus
                    \left\{{\begin{pmatrix}
                       0 & {}^tc_1 & 0 \\
                       0 & \mathrm{O}_n & c_2 \\
                       0 & 0 & 0
                    \end{pmatrix}}\right\}
                       \oplus
                    \left\{{\begin{pmatrix}
                       0 & 0 & d \\
                       0 & \mathrm{O}_n & 0 \\
                       0 & 0 & 0
                     \end{pmatrix}}\right\}, \\
  (a,d,\alpha,\beta \in {\mathbf R},\,  b_1,b_2,c_1,c_2\in
                    {\mathbf R}^n, \, A\in \mathfrak{gl}(n,{\mathbf{R}}); \,
                    \alpha +\beta +\operatorname{tr}A=0),
                  \qquad [\mathfrak{g}_p,\mathfrak{g}_q] \subset
                    \mathfrak{g}_{p+q}.
\end{gather*}
Put
\begin{gather*}
\mathfrak{m}= \mathfrak{g}_{-2}\oplus \mathfrak{g}_{-1},
\end{gather*}
then $\mathfrak{m}$ is a fundamental GLA of contact type, i.e., Heisenberg
algebra.
Put
\begin{gather*}
\mathfrak{g}'=\mathfrak{g}_0 \oplus \mathfrak{g}_1 \oplus
\mathfrak{g}_2.
\end{gather*}
Let $G'$ be the Lie subgroup of $G = {\rm PGL}(n+2, {\mathbf{R}})$ def\/ined by
\begin{gather*}
G' = P\left\{
\begin{pmatrix}
* & * & * \\
0 & * & * \\
0 & 0 & *
\end{pmatrix}
\in
{\rm GL}(n+2, {\mathbf{R}})
\right\}.
\end{gather*}
Then the Lie algebra of $G'$ is given by $\mathfrak{g}'$. Note that
$\dim \mathfrak{g}' = n^2 + 2n + 2$.

The group $G$ transitively acts on the f\/lag manifold
\begin{gather*}
P\big(T^*P^{n+1}\big) = \big\{ V_1 \subset V_{n+1} \subset {\mathbf{R}}^{n+2} \,|\,
\dim V_1 = 1, \dim V_{n+1} = n + 1 \big\} \subset P^{n+1}\times P^{n+1*}.
\end{gather*}
Then $G'$ is the isotropy group of $(V_1, V_{n+1}) = (\langle e_0 \rangle,
\langle e_0, \dots, e_n \rangle)$ for the standard basis $e_0, e_1, \dots$,
$e_n, e_{n+1}$ of ${\mathbf{R}}^{n+1}$.
Therefore we have
\begin{gather*}
G/G' \cong P\big(T^*P^{n+1}\big) \qquad \big(\cong P\big(T^*P^{n+1*}\big)\big).
\end{gather*}
Note that $\mathfrak{g}_{-1} \subset \mathfrak{m} = T_o(G/G')$, where
$o = G'$ is the origin,
def\/ines the contact structure $D$ on $G/G'$ which corresponds to
the canonical contact structure
on $P(T^*P^{n+1})$ via the above dif\/feomorphism (cf.~\cite{I-M}).

Next, we consider
\begin{gather*}  \mathfrak{e}^1=\left\{{\begin{pmatrix}
                      0 & 0 & 0 \\
                      b_1 & \mathrm{O}_n & 0 \\
                      0 & 0 & 0
                    \end{pmatrix}}\right\}, \;
  \mathfrak{e}^2=\left\{{\begin{pmatrix}
                      0 & 0 & 0 \\
                      0 & \mathrm{O}_n & 0 \\
                      0 & {}^tb_2 & 0
                     \end{pmatrix}}\right\}.
\end{gather*}
Then we have
\begin{gather*}
\mathfrak{g}_{-1}=\mathfrak{e}^1\oplus \mathfrak{e}^2,\qquad
[\mathfrak{e}^1,\mathfrak{e}^1]=[\mathfrak{e}^2,\mathfrak{e}^2]=0,\qquad
\mathfrak{g}_{-2}=[\mathfrak{e}^1,\mathfrak{e}^2],
\\
[\mathfrak{g}_0,\mathfrak{e}^1]
       \subset \mathfrak{e}^1,\qquad
       [\mathfrak{g}_0,\mathfrak{e}^2]
       \subset \mathfrak{e}^2.
\end{gather*}
Denote by $E_{ij}\in \mathfrak{gl}(n+2,{\mathbf{R}})$ the matrix unit of
$(i,j)$ component. Then we put
\begin{gather*}
\gamma =E_{n+2,1}\in \mathfrak{g}_{-2},\\
e_i =E_{i+1,1}\in
\mathfrak{e}^1\subset \mathfrak{g}_{-1},\qquad f_i=E_{n+2,i+1}\in
\mathfrak{e}^2\subset \mathfrak{g}_{-1}\qquad (1\leq i\leq n).
\end{gather*}
We set
\begin{gather*}
[X,Y] = - A(X,Y)\gamma \qquad (X,Y\in \mathfrak{g}_{-1}).
\end{gather*}
It follows that
\begin{gather*}
A(e_i,f_j)={\delta}_{ij}, \qquad A(e_i, e_j) = 0, \qquad A(f_i, f_j) = 0 \qquad  (1 \leq i, j \leq n).
\end{gather*}
Therefore we have that $A$ is a symplectic form.
Then $\mathfrak{g}_{-1}$ becomes a symplectic vector space with respect to~$A$,  and
$e_1, \dots, e_n$, $f_1, \dots, f_n$ form a symplectic basis of the symplectic vector space~$(\mathfrak{g}_{-1},A)$.
Moreover
$\mathfrak{e}^1$, $\mathfrak{e}^2$ form a Lagrangian pair of~$(\mathfrak{g}_{-1}, A)$.

Moreover, put
\begin{gather*}
\mathfrak{a}^1=\mathfrak{e}^1+\mathfrak{g}', \qquad
\mathfrak{a}^2=\mathfrak{e}^2+\mathfrak{g}'.
\end{gather*}
Then we easily verify that
\begin{gather*}
\mathrm{Ad}(G')\mathfrak{a}^1=\mathfrak{a}^1, \qquad
\mathrm{Ad}(G')\mathfrak{a}^2=\mathfrak{a}^2,\qquad
\mathfrak{a}^1\cap \mathfrak{a}^2=\mathfrak{g}'.
\end{gather*}
Thus $\mathfrak{a}^1$ and $\mathfrak{a}^2$ induce invariant
dif\/ferential systems~$E_1$ and $E_2$ on $G/G'$ respectively.
The pair $(E_1,E_2)$ forms
a Lagrangian pair of $(G/G', D)$ and thus
that of the standard contact structure $D$ on $P(T^*P^{n+1})$.
Moreover
both $E_1$ and $E_2$ are completely integrable.
It follows that
$P(T^*P^{n+1})/\mathcal{E}_2\cong P^{n+1}$, $P(T^*P^{n+1})/\mathcal{E}_1
\cong {P^{n+1}}^*$,
where
$\mathcal{E}_1$, $\mathcal{E}_2$ are the foliations induced by~$E_1$,~$E_2$
respectively.

For the linear isotropy representation
$\rho \colon  G' \longrightarrow \mathrm{GL}(\mathfrak{m})$ at the origin
$o = G'$ in $G/G'$,
we have, with respect to the basis $\{ \gamma, e_1, \dots, e_n, f_1, \dots, f_n \}$ in $\mathfrak{m}$,
\begin{gather*}
\tilde{G} =\rho(G')
          =\left\{{\begin{pmatrix}
                      a & 0 & 0 \\
                      b_1 & A & \mathrm{O}_n \\
                      b_2 & \mathrm{O}_n & a\, {}^t\! A^{-1}
                    \end{pmatrix}} \Bigg\vert \,
a\in \mathbf{R}^*,\, A\in \mathrm{GL}(n,{\mathbf{R}}), \,b_1, b_2\in \mathbf{R}^n
\right\}.
\end{gather*}

Then the $\tilde{G}$-structures of type $\mathfrak{m}$ are in bijective
correspondence with the Lagrangian contact structures~\cite[Theorem~5.1]{Tak}.
Note that $\tilde{G}$-structure is of inf\/inite type~\cite[Chapter~I]{K}.
How\-ever~${\mathfrak g}$ is the prolongation of $({\mathfrak m}, {\mathfrak g}_0)$
in the sense of Tanaka~\cite{Tan1,Y}. Moreover
$\tilde{G} = G_0^{\#}$ in the notation of~\cite{Tan}.
Thus, by the f\/initeness theorem of Tanaka~\cite[Corollary~2]{Tan},  we have:

\begin{Prop}
Let $(M,D)$ be a contact manifold of dimension $2n+1$
with a Lagrangian pair $(E_1,E_2)$.
Then the automorphism pseudo-group of all compatible
diffeomorphisms on $M$
as a Lagrangian contact structure
is of finite type, that is to say,
it is a finite-dimensional Lie pseudo-group.
The maximum dimension of the automorphism pseudo-groups, fixing~$n$, is given by
$\dim \mathfrak{sl}(n+2,{\mathbf{R}}) = (n+2)^2-1$
which the flat model attains.
\end{Prop}

Moreover, using a result in \cite{M-M}, we have:

\begin{Prop}
\label{Hess=0}
The equivalence class of a decomposable Monge--Amp{\`e}re system
with a~Lag\-rangian pair $(E_1, E_2)$ on a contact manifold $(M, D)$
is uniquely determined by the Lagrangian contact structure
$(D, E_1, E_2)$.
Therefore the maximal dimension of the automorphism pseudo-groups of decomposable Monge--Amp{\`e}re systems
on $(M^{2n+1}, D)$ is equal to $(n+2)^2-1$.
The maximum is attained by the decomposable Monge--Amp{\`e}re system ${\rm Hess} = 0$
on the flat model on $M = PT^*(P^{n+1})$.
\end{Prop}

In fact it is given the following result in~\cite{M-M}, which implies Proposition~\ref{Hess=0}:

\begin{Lem}[\protect{\cite[Proposition~2.1]{M-M}}]
\label{decomposable-M-M}
Let $(V, \Theta)$ be a symplectic vector space of dimension~$2n$.
For given two non-zero decomposable $n$-covectors $\omega = \beta_1 \wedge \cdots \wedge \beta_n$,
$\omega' = \beta'_1 \wedge \cdots \wedge \beta'_n$
with $\beta_i, \beta'_i \in V^*$,
we have $\omega' = \lambda\omega + \phi\wedge\Theta$ for a nonzero scalar $\lambda$ and a~$(n-2)$-covector $\phi$, if and only if, the annihilator of $\beta_1, \dots, \beta_n$ in~$V$ and
the annihilator of $\beta'_1, \dots, \beta'_n$ in~$V$ are either identical
or perpendicular with respect to~$\Theta$.
\end{Lem}

\begin{proof}[Proof of Proposition \ref{Hess=0}]
Let ${\mathcal M}$ be a decomposable Monge--Amp{\`e}re system
with a~Lag\-rangian pair $(E_1, E_2)$ on $(M, D)$.
First we observe that ${\mathcal M}$ has a decomposable local generator.
To see this, let $X_1, \dots, X_n$ and $P_1, \dots, P_n$ be local frames
of $E_1$ and~$E_2$ respectively.
Let~$R$ be the Reeb vector f\/ield
for a local contact form~$\theta$ def\/ining~$D$.
Consider the dual coframe $\theta, \alpha_1,\dots,\alpha_n, \beta_1,\dots,\beta_n$
of~$T^*M$ to the frame $R, X_1, \dots, X_n, P_1, \dots, P_n$ of $TM$.
Then we see, by the decomposing condition,
that there exist an $(n-1)$-form $\gamma$ and a
non-vanishing function~$\mu$ on~$M$ such that
$\omega = \mu(\beta_1 \wedge \cdots \wedge \beta_n) + \theta\wedge\gamma$.
Thus we have ${\mathcal M} = \langle \beta_1 \wedge \cdots \wedge \beta_n, \theta, d\theta\rangle$.
Note that ${\rm Ann}(\theta, \alpha_1, \dots, \alpha_n) = E_1$.
Then, by Lemma~\ref{decomposable-M-M}, we see that ${\mathcal M}$ is determined just by~$E_1$.
In particular, given a Lagrangian pair $(E_1, E_2)$, the decomposable Monge--Amp{\`e}re system
with the Lagrangian pair $(E_1, E_2)$ is uniquely determined.
\end{proof}

\begin{Rem}
If the characteristic system $E_1$ is integrable, then the decomposable Monge--Amp{\`e}re system
is isomorphic to the system corresponding to the equation ${\rm Hess} = 0$ (see \cite{M-M,Mo2}).
\end{Rem}

\begin{Rem}
The equation ${\rm Hess} = 0$ has the inf\/inite-dimensional automorphism pseudo-group
which consists of the lifts of dif\/feomorphisms on the dual projective space.
However, if a Lagrangian pair is associated, then the automorphism pseudo-group
turns to be of f\/inite-dimensional, as stated in Proposition~\ref{Hess=0}.
\end{Rem}

\section[Automorphisms of Monge-Amp{\`e}re systems with Lagrangian pairs]{Automorphisms of Monge--Amp{\`e}re systems\\ with Lagrangian pairs}
\label{Automorphisms}

Let us consider the automorphism pseudo-group Aut($\mathcal{M}$) of
all local isomorphisms of a bi-decomposable Monge--Amp{\`e}re system
$\mathcal{M} = \langle \theta, d\theta, \omega \rangle$
with a Lagrangian pair $(E_1,E_2)$
on a~contact manifold $(M,D)$ of dimension $2n+1$.

By Theorems \ref{bi-decomposable-form} and~\ref{Uniqueness}, in the case $n \geq 3$,
any automorphism of any bi-decomposable Monge--Amp{\`e}re system with a~Lagrangian pair $(E_1, E_2)$ preserves the Lagrangian pair $(E_1, E_2)$
up to the interchange of~$E_1$,~$E_2$.
Therefore inf\/initesimal symmetries
of bi-decom\-po\-sable Monge--Amp{\`e}re systems
are studied based on inf\/initesimal symmetries of Lagrangian contact structures.

Now let $\overline{G}$ be
\begin{gather*}
\overline{G} = \left\{\left. \begin{pmatrix}
                       c^2 & 0 & 0 \\
                        b_1 & cA &   \mathrm{O}_n \\
                        b_2 &   \mathrm{O}_n &   c\, {}^t\!\! A^{-1}
                    \end{pmatrix}
                    \, \right\vert \,
c \in {\mathbf{R}}^{\times}, \,A\in \mathrm{SL}(n,{\mathbf{R}}),\, b_1, b_2 \in {\mathbf{R}}^n
\right\}.
\end{gather*}

\begin{Prop}
The bi-decomposable Monge--Amp{\`e}re systems with Lagrangian pairs
are in bijective correspondence with $\overline{G}$-structures of type~$\mathfrak{m}$.
\end{Prop}

\begin{proof}
We consider, as in Section~\ref{Lagrangian contact structures.},
the fundamental GLA of contact type
$\mathfrak{m}= \mathfrak{g}_{-2}\oplus \mathfrak{g}_{-1}$ and
the decomposition
$\mathfrak{g}_{-1}=\mathfrak{e}^1\oplus \mathfrak{e}^2$ into the Lagrangian pair.
Moreover we f\/ix a volume form $\Omega_1 \in \wedge^n(\mathfrak{e}^{1*})$ on $\mathfrak{e}^1$ and
a volume form $\Omega_2 \in \wedge^n(\mathfrak{e}^{2*})$ on $\mathfrak{e}^2$.
Consider the group $C(\mathfrak{m}; \mathfrak{e}^1, \Omega_1; \mathfrak{e}^2, \Omega_2)$
consisting of all $a \in  {\rm GL}(\mathfrak{m})$ which satisf\/ies the following conditions:
$a\mathfrak{g}_{-1} = \mathfrak{g}_{-1}$,
the graded linear automorphism $\overline{a}$ of $\mathfrak{m}$ induced by $a$ is a GLA-automorphism,
$a\mathfrak{e}^1 = \mathfrak{e}^1$, $a\mathfrak{e}^2 = \mathfrak{e}^2$, and
$a^*\Omega_1 = \lambda\Omega_1$, $a^*\Omega_2 = \lambda\Omega_2$ for some $\lambda \in {\mathbf{R}}^{\times}$.
Then we have that $C(\mathfrak{m}; \mathfrak{e}^1, \Omega_1; \mathfrak{e}^2, \Omega_2)$ is identical with
$\overline{G}$.
\end{proof}

Thus the equivalence problem of bi-decomposable Monge--Amp{\`e}re systems with Lagrangian pairs
is studied as an adapted $G$-structure on a contact manifold of dimension~$2n+1$.

Let $G^0$ be a subgroup of ${\rm GL}(n+2, {\mathbf{R}})$ def\/ined by
\begin{gather*}
G^0 = \left\{ \left.
\begin{pmatrix}
                      \ k^{-1} & 0 & \ 0 \\
                      \ b_1 & A & \ 0 \\
                      \ a & {}^tb_2 &  \ k
                    \end{pmatrix}
                   \,  \right\vert \,
                    k \in {\mathbf{R}}^{\times}, \, A \in  {\rm SL}(n, {\mathbf{R}}), \, b_1, b_2 \in {\mathbf{R}}^n, \, a \in {\mathbf{R}}
                    \right\}.
\end{gather*}
We set
\begin{gather*}
H^0 = \left\{ \left.
\begin{pmatrix}
                      \ k^{-1} & 0 & \ 0 \\
                      \ 0 & A & \ 0 \\
                      \ 0 & 0 &  \ k
                              \end{pmatrix}
         \,  \right\vert   \,
        k \in {\mathbf{R}}^{\times}, \, A \in  {\rm SL} (n, {\mathbf{R}})
 \right\}
\subset G^0.
\end{gather*}
Then we have the model space
$G^0/H^0 \cong {\mathbf{R}}^{2n+1}$ with coordinates
$(x, z, p)$.
The $G^0$-action on ${\mathbf{R}}^{2n+1}$ is described by
\begin{gather*}
(x, z, p) \mapsto (x', z', p') = \left(k(Ax + b_1), k\big(kz - {}^tb_2x - a\big), k\left(p - \frac{1}{k}{}^tb_2\right)A^{-1}\right).
\end{gather*}
Then $dz' - p'dx' = k^2(dz - pdx)$ and
the form $\omega = c dx_1\wedge \cdots\wedge dx_n - dp_1\wedge\cdots\wedge dp_n$
is transformed to $\omega' = k^n\omega$.

Let $E_{i,j}$, $0 \leq i, j \leq n+1$, denotes the elementary $(n+2)\times (n+2)$-matrix such that
the $(i, j)$ component is $1$ and other components are all zero.
We set
\begin{gather*}
\varepsilon = - E_{0,0} + E_{n+1,n+1}, \qquad e_i = E_{i,0} \quad (1 \leq i \leq n), \\
f_j = E_{n+1,j} \quad (1 \leq j \leq n), \qquad \gamma = E_{n+1,0},
\end{gather*}
and identify the set of traceless matrices ${\mathfrak{sl}}(n, {\mathbf{R}})$ with
\begin{gather*}
\left\{ \left.
\begin{pmatrix}
                      0 & 0 &  0 \\
                       0 & A &  0 \\
                       0 & 0 &  0
                              \end{pmatrix}
         \,  \right\vert   \,
       A \in {\mathfrak{sl}}(n, {\mathbf{R}})
 \right\}
\subset {\mathfrak g}^0.
\end{gather*}
Then we set
${\mathfrak g}_0 = \langle\varepsilon\rangle_{{\mathbf{R}}} \oplus {\mathfrak{sl}}(n, {\mathbf{R}})$,
${\mathfrak g}_{-1}^1 = \langle e_1, \dots, e_n\rangle_{{\mathbf{R}}}$,
${\mathfrak g}_{-1}^2 = \langle f_1, \dots, f_n\rangle_{{\mathbf{R}}}$,
${\mathfrak g}_{-1} = {\mathfrak g}_{-1}^1 \oplus {\mathfrak g}_{-1}^2$,
and ${\mathfrak g}_{-2} = \langle\gamma\rangle_{{\mathbf{R}}}$.
Then
\begin{gather*}
{\mathfrak g}^0 = {\mathfrak g}_{-2} \oplus {\mathfrak g}_{-1} \oplus {\mathfrak g}_{0}
\end{gather*}
is the Lie algebra of $G^0$. We write
${\mathfrak m} = {\mathfrak g}_{-2}\oplus{\mathfrak g}_{-1}$.
Then the Lie subalgebra~${\mathfrak m}$ becomes the split Heisenberg algebra.

The matrix representation of $A_0 \in \mathfrak{g}_0=\langle\varepsilon\rangle_{{\mathbf{R}}} \oplus {\mathfrak{sl}}(n, {\mathbf{R}})$
with respect to the basis $\{ \gamma, \, e_i $ $(1 \leq i \leq n),\,   f_j\,  (1 \leq i \leq n) \}$ in
the Heisenberg algebra
$V= \mathfrak{m} = \mathfrak{g}_{-2}\oplus \mathfrak{g}_{-1}$
has the following form:
\begin{gather*}
A_0=C+A_0'=
 c
\begin{pmatrix}
2 & 0 & 0 \\
0 & I & \mathrm{O} \\
0 & \mathrm{O} & I
\end{pmatrix}
+
\begin{pmatrix}
0 & 0 & 0 \\
0 & A & \mathrm{O} \\
0 & \mathrm{O} & -{}^t\!A
\end{pmatrix}.
\end{gather*}
In fact, for the commutators, we have by the direct calculations:
\begin{gather*}
[\varepsilon, \gamma] = 2\gamma,
\qquad
[\varepsilon, e_i] = e_i,
\qquad
[\varepsilon, f_j] = f_j,
\\
[A, \gamma] = O,
\qquad
[A, e_i] = a_i,
\qquad
[A, f_j] = - a^j,
\end{gather*}
where $a_i$ is the $i$-th column of $A$ and $a^j$ is the $j$-th row of $A$,
$1 \leq i, j \leq n$, and $A \in {\mathfrak{sl}}(n, {\mathbf{R}})$.
Moreover we have
\begin{gather*}
[e_i, f_j] = - \delta_{ij}\gamma, \qquad [e_i, e_j] = 0, \qquad [f_i, f_j] = 0 \quad (1 \leq i, j \leq n).
\end{gather*}

We will study the prolongation of
$({\mathfrak m}, {\mathfrak g}_0)$.
We def\/ine the prolongation inductively
${\mathfrak g}_k = {\mathfrak g}({\mathfrak m}, {\mathfrak g}_0)_k$ $(k \geq 1)$ by
the set of elements
$\{
(\alpha, \beta) \in  {\rm Hom}({\mathfrak g}_{-1}, {\mathfrak g}_{k-1})\oplus
{\rm Hom}({\mathfrak g}_{-2}, {\mathfrak g}_{k-2})\}$ satisfying
\begin{alignat*}{3}
& {\rm (i)} \ \ && \beta([x, y]) = [\alpha(x), y] - [\alpha(y), x]   \qquad (x, y \in {\mathfrak g}_{-1}), &
\\
& {\rm (ii)} \ \ && [\alpha(y), z] = [\beta(z), y]  \qquad (y \in {\mathfrak g}_{-1},\, z \in {\mathfrak g}_{-2}).&
\end{alignat*}
See \cite[p.~429]{Y}. Then we have

\begin{Lem}
\label{prolongation-calculation}
The prolongation ${\mathfrak g}_i$ vanishes for any $i \geq 1$.
\end{Lem}

\begin{proof}
First we calculate ${\mathfrak g}_1$. For any $(\alpha, \beta) \in
{\rm Hom}(\mathfrak{g}_{-1}, \mathfrak{g}_0) \oplus {\rm Hom}(\mathfrak{g}_{-2}, \mathfrak{g}_{-1})$,
the condi\-tions~(i) and (ii) imply that
\begin{alignat*}{3}
& \text{(i)}   \ \ && \beta([e_i, f_j]) = [\alpha(e_i), f_j] - [\alpha(f_j), e_i]  \qquad (1 \leq i, j \leq n), &
\\
& \text{(ii-1)} \ \ & & [\alpha(e_i), \gamma] = [\beta(\gamma), e_i] \qquad (1 \leq i \leq n),&
\\
& \text{(ii-2)} \ \ && [\alpha(f_j), \gamma] = [\beta(\gamma), f_j] \qquad (1 \leq j \leq n).&
\end{alignat*}
Set
\begin{gather*}
\beta(\gamma) = \sum_{\ell=1}^n  b_\ell e_\ell + \sum_{m=1}^n  c_m  f_m,
\qquad
\alpha(e_i) = h_i\varepsilon + A^{(i)}, \qquad
\alpha(f_j) = k_j\varepsilon + B^{(j)},
\end{gather*}
where $b_\ell, c_m, h_i, k_j \in {\mathbf{R}}$, $A^{(i)}, B^{(j)} \in {\mathfrak{sl}}(n, {\mathbf{R}})$.
Then
\begin{gather*}
\beta([e_i, f_j]) = \sum_{\ell=1}^n (- \delta_{ij}b_\ell)e_\ell + \sum_{m=1}^n (- \delta_{ij}c_m)f_m,
\\
{[\alpha(e_i), f_j]} = h_i f_j - \big(j{\mbox{\rm -th row of }} A^{(i)}\big), \qquad
{[\alpha(f_j), e_i]} = k_j e_i + \big(i{\mbox{\rm -th column of }} B^{(j)}\big),
\\
{[\alpha(e_i), \gamma]} = 2h_i\gamma, \qquad {[\beta(\gamma), e_i]} = c_i\gamma, \qquad
{[\alpha(f_j), \gamma]} = 2k_j\gamma, \qquad {[\beta(\gamma), f_j]} = - b_j\gamma.
\end{gather*}
By the condition (i), we have
\begin{gather*}
B^{(j)}_{\ell i} = - \delta_{\ell i}k_j + \delta_{ij}b_\ell, \qquad
A^{(i)}_{j m} = \delta_{j m}h_i + \delta_{ij}c_m \qquad (1 \leq i, j, \ell, m \leq n).
\end{gather*}
By the condition (ii-1), we have $c_i = 2h_i$, $1 \leq i \leq n$.
By the condition (ii-2), we have $b_j = - 2k_j$, $1 \leq j \leq n$.
Therefore we have
\begin{gather*}
B^{(j)}_{\ell i} = - \delta_{\ell i}k_j - 2\delta_{ij}k_\ell, \qquad
A^{(i)}_{j m} = \delta_{j m}h_i + 2\delta_{ij}h_m \qquad (1 \leq i, j, \ell, m \leq n).
\end{gather*}
In particular, for any $j$, $\ell$ with $1 \leq j, \ell \leq n$,
we have $B^{(j)}_{\ell\ell} = - k_j - 2\delta_{\ell j}k_\ell$, and therefore
\begin{gather*}
0 = \operatorname{tr} B^{(j)} = - (n+2)k_j,
\end{gather*}
hence $k_j = 0$, $1 \leq j \leq n$.
For any $i$, $m$ with $1 \leq i, m \leq n$, we have
$A^{(j)}_{mm} = h_i + 2\delta_{i m}h_m$, and therefore
\begin{gather*}
0 = \operatorname{tr}  A^{(i)} = (n+2)h_i,
\end{gather*}
hence $h_i = 0$, $1 \leq i \leq n$.
Thus we have $\beta = 0$ and $\alpha = 0$. Therefore we have ${\mathfrak g}_1 = 0$.

Second we calculate ${\mathfrak g}_2$. Take any $(\alpha, \beta) \in {\mathfrak g}_2 \subset
{\rm Hom}({\mathfrak g}_{-1}, {\mathfrak g}_{1})\oplus
{\rm Hom}({\mathfrak g}_{-2}, {\mathfrak g}_{0})$.
Since ${\mathfrak g}_{1} = 0$, $\alpha = 0$. Then we see $\beta(\gamma) = 0$. Therefore $\beta = 0$.
Thus we obtain that ${\mathfrak g}_{2} = 0$.

From ${\mathfrak g}_{1} = 0$, ${\mathfrak g}_{2} = 0$, we have ${\mathfrak g}_{i} = 0$, $i \geq 3$, automatically.
\end{proof}

Then we have

\begin{Th}
\label{maximal symmetry}
Let $(M, D)$ be a contact manifold of dimension $2n+1$
and~$\mathcal{M}$
a bi-decomposable Monge--Amp{\`e}re system
with a Lagrangian pair on~$(M,D)$.
Assume that $n \geq 3$.
Then the automorphism pseudo-group $\operatorname{Aut}(\mathcal{M})$
of~$\mathcal{M}$ has
at most the dimension of
$(n+1)^2$. The estimate is best possible.
\end{Th}

\begin{proof}
By Lemma \ref{prolongation-calculation},
we have known that the prolongations ${\mathfrak g}_i$, $1 \leq i$ in the sense of Tanaka vanish.
On the other hand $\overline{G}$ is equal to $(H^0)^{\#}$ in the notation of~\cite{Tan}.
Then, by the f\/initeness theorem of Tanaka~\cite[Corollary~2]{Tan}, we see that
the dimension of the pseudo-group of automorphisms on ${\mathcal M}$ is estimated by
$\dim({\mathfrak g^0}) = \sum\limits_{i \leq 0} \dim({\mathfrak g}_i) = (n+1)^2$.
Moreover there exists an Monge--Amp{\`e}re system $\mathcal{M}$ with a Lagrangian pair on
$M = {\mathbf{R}}^{2n+1}$, which arises in equi-af\/f\/ine geometry, such that
the automorphism group {\rm Aut}($\mathcal{M}$) attains the maximal dimension $(n+1)^2$.
See Section~\ref{Homogeneous M-A systems with Lagrangian pairs.}.1.
\end{proof}

\begin{Rem}
The sharp symmetry bounds of non-f\/lat Lagrangian contact structures together with many other parabolic
geometries have been obtained by Kruglikov and The~\cite{K-T}.
Lemma~\ref{prolongation-calculation} in our paper
is similar to the concept of prolongation rigidity studied by them in~\cite{K-T}.
\end{Rem}

\section{Hesse representations}
\label{Hesse representations.}

Let $(E_1,E_2)$ be a Lagrangian pair on a contact manifold
$(M, D)$. We call $(E_1,E_2)$
{\it bi-Legendre-integrable} or simply {\it integrable} if
both~$E_1$ and~$E_2$
are completely integrable as subbundles in~$TM$.
Then $(E_1, E_2)$ def\/ines a pair of Legendrian foliations
$({\mathcal E}_1, {\mathcal E}_2)$ on $M$ locally.
The standard Lagrangian pair $(E^{\rm{st}}_1, E^{\rm{st}}_2)$ introduced in
Section~\ref{Monge-Ampere systems and Lagrangian pairs}
is integrable. In fact the foliation~${\mathcal E}^{\rm{st}}_1$
is def\/ined by f\/ibers of the projection
$(x, z, p) \mapsto \Big(p, \sum\limits_{i=1}^n x_ip_i - z\Big)$
and the foliation ${\mathcal E}^{\rm{st}}_2$
is def\/ined by $(x, z, p) \mapsto (x, z)$.

\begin{Def}
If a Lagrangian pair $(E_1,E_2)$ on $(M, D)$
is locally contactomorphic to the standard Lagrangian pair
$(E^{\rm{st}}_1, E^{\rm{st}}_2)$ on $({\mathbf{R}}^{2n+1}, D_{\rm{st}})$,
then we call $(E_1, E_2)$ {\it flat}.
\end{Def}

Let $(E_1, E_2)$ be a Lagrangian pair
on a contact manifold $(M, D)$ of dimension $2n + 1$.
Assume that $(E_1, E_2)$ is bi-Legendre-integrable.
Then, locally, there exist Legendrian f\/ibrations
\mbox{$\pi_1 \colon  M \to W_1$}
and $\pi_2 \colon  M \to W_2$ having
$E_2$ and $E_1$ as the kernels
of the dif\/ferentials $(\pi_1)_* \colon  TM
\to TW_1$
and $(\pi_2)_* \colon  TM
\to TW_2$
for some manifolds $W_1$ and $W_2$ of dimension $n+1$
respectively.
Then we have the following diagram:
\begin{gather*}
\xymatrix{
& M \ar[ld]_{\pi_1} \ar[rd]^{\pi_2}  &\\
W_1 &  & W_2}
\end{gather*}
We call $(\pi_1, \pi_2)$ a {\it
double Legendrian fibration}.
In this case we say that $W_1$ and $W_2$ are in
the {\it dual} relation via the {\it Legendre transformation} on~$M$.
Since $D = E_1\oplus E_2$,
we see $(\pi_1, \pi_2)\colon M \to W_1\times W_2$
is an immersion. Thus $M$ has a {\it pseudo-product structure}
doubly foliated by Legendrian submanifolds~\cite{Tab, Tan1}.

\begin{Ex}
The projective cotangent bundle
$M = P(T^*W)$ over a manifold $W$ of dimension $n+1$
with the canonical projection $\pi_1 \colon M \to W$
has the canonical contact structure $D \subset TM$
def\/ined by, for any $(x, [\alpha]) \in M$ with $x \in W$, $[\alpha] \in P(T^*_xW)$,
\begin{gather*}
D_{(x, [\alpha])} = \{ v \in T_xM \,|\, \alpha((\pi_1)_* v) = 0 \}.
\end{gather*}
We see moreover that $D$ has the integrable Lagrangian subbundle
$E_2 = {\rm Ker}(\pi_{1*})$.

Assume that $W$ is the projective space~$P^{n+1}$.
Then the contact manifold
$M=P(T^*P^{n+1})$ is naturally identif\/ied with the contact manifold
$P(T^*{P^{n+1}}^*)$, the projective cotangent bundle over
the dual projective space ${P^{n+1}}^*$ with
the canonical projection $\pi_2 \colon P(T^*{P^{n+1}}^*)
\to {P^{n+1}}^*$~\cite{I-Mo}.
Then $D \subset TM$ has another Lagrangian subbundle
$E_1 = {\rm Ker}((\pi_2)_*)$ and $(E_1, E_2)$ turns to be an
integrable Lagrangian pair of $(M, D)$.
There is associated to $(E_1, E_2)$ the double Legendrian f\/ibration
\begin{gather*}
\xymatrix{
& P\big(T^*P^{n+1}\big) \ar[ld]_{\pi_1} \ar[rd]^{\pi_2}  &\\
P^{n+1} &  & {P^{n+1}}^*}
\end{gather*}
for $W_1=P^{n+1}$, $W_2= {P^{n+1}}^*$.
This globally def\/ined Lagrangian pair $(E_1, E_2)$ is f\/lat.
In fact, $P(T^*P^{n+1})$ is identif\/ied with
the incidence hypersurface $I \subset P^{n+1}\times {P^{n+1}}^*$ def\/ined
by
\begin{gather*}
x_0y_0 + x_1y_1 + \cdots + x_{n+1}y_{n+1} = 0
\end{gather*}
for homogeneous coordinates $[x] = [x_0 : x_1 : \cdots : x_{n+1}]$
of $P^{n+1}$ and
$[y] = [y_0 : y_1 : \cdots : y_{n+1}]$
of ${P^{n+1}}^*$. On the af\/f\/ine open subset $U =
\{ x_0 \not= 0, \, y_{n+1}\not= 0\} \subset P^{n+1}\times {P^{n+1}}^*$,
take local coordinates of
\begin{gather*}
x'_i = \frac{x_i}{x_0},
\qquad
z' = - \frac{x_{n+1}}{x_0},
\qquad
p'_j = \frac{y_j}{y_{n+1}},
\qquad
\tilde{z}' = - \frac{y_0}{y_{n+1}},
\end{gather*}
$(1 \leq i \leq n,\ 1 \leq j \leq n)$.
Then $I \cap U$ is def\/ined by
$- z' - \tilde{z}' + \sum\limits_{i=1}^n x'_ip'_i = 0$, where we have
$dz' - \sum\limits_{i=1}^n p'_i dx'_i + \Big(d\tilde{z}' - \sum\limits_{i=1}^n x'_i dp'_i\Big) = 0$.
The contact structure on $I \cap U$ is given by
$dz' - \sum\limits_{i=1}^n p'_i dx'_i = 0$.
Also it is given by $d\tilde{z}' - \sum\limits_{i=1}^n x'_i dp'_i = 0$.
This shows $(E_1, E_2)\vert_{U}$ is f\/lat. Also on other af\/f\/ine open subsets,
we can verify the f\/latness of $(E_1, E_2)$ similarly or by
an argument using homogeneity.
\end{Ex}

\begin{Ex}
\label{hor+ver}
For a Riemannian manifold $W=(W,g)$ of dimension $n+1$,
the unit tangent bundle $M=T_1W$ of $W$ has
the canonical contact structure $D\subset TM$,
\begin{gather*}
D_{(x, v)} = \big\{ u \in T_{(x, v)}M \,|\, \pi_*u \in v^{\perp}\big\}, \qquad (x, v) \in T_1W,
\end{gather*}
and
the canonical Lagrangian pair $(E_1,E_2)$ induced by
the horizontal lift and the vertical lift of the
Levi-Civita connection:
\begin{gather*}
(E_1)_{(x, v)} = \big(v^\perp\big)^{\rm{hor}}, \qquad (E_2)_{(x, v)} =
\big(v^\perp\big)^{\rm{ver}}, \qquad (x, v) \in T_1W.
\end{gather*}
Here $\pi\colon M \longrightarrow W$ is the canonical projection
and $v^{\perp} = \{ w\in T_{x}W \,|\, g(v,w)=0 \}$.
Then $(E_1,E_2)$ is bi-Legendre-integrable if and only if~$W$ is a space form.
In fact, the vertical lift~$E_2$ is always completely integrable.
Moreover the horizontal lift $E_1$ is completely integrable
if and only if~$W$ is projectively f\/lat, that is, a space form (see~\cite[Corollaries~3.5 and~6.5]{Tak}).
\end{Ex}

The bi-decomposable class of Monge--Amp{\`e}re systems with f\/lat Lagrangian pairs
turns out to be an intrinsic representation of
the well-known class of Monge--Amp{\`e}re equations
locally expressed by Hesse representations:
\begin{gather*}
\mathrm{Hess}(z) = F(x_1, \dots, x_n, z, p_1, \dots, p_n)\quad (\not=0).
\end{gather*}
Thus we are led to the following def\/inition:

\begin{Def}
\label{Hesse-M-A-def}
We call a Monge--Amp{\`e}re system with a f\/lat Lagrangian pair
a {\it Hesse Monge--Amp{\`e}re system}.
\end{Def}

Note that there are sub-classes of Hesse Monge--Amp{\`e}re equations,
{\it Euler--Lagrange Monge--Amp{\`e}re equations}:
\begin{gather*}
\mathrm{Hess}(z) = F_1(x_1, \dots, x_n, z)\cdot
F_2\left(p_1,\dots, p_n, \sum_{j=1}^n p_jx_j-z\right)\quad (\not=0),
\end{gather*}
(see \cite[p.~21, Example~2]{B-G-G}) and {\it flat Monge--Amp{\`e}re equations}:
\begin{gather*}
\mathrm{Hess}(z) = c\quad (c \ \text{is constant}, \ c\not=0).
\end{gather*}

\begin{Def}
\label{E-L-M-A-def}
Let $\mathcal{M}$ be a Hesse Monge--Amp{\`e}re system.
Then we call $\mathcal{M}$ an {\it Euler--Lagrange Monge--Amp{\`e}re system}
if $\mathcal{M}$ is locally generated by a bi-decomposable
$n$-form $\omega = \omega_1 - \omega_2$ satisfying
\begin{gather*}
d\omega_1 \equiv 0, \qquad d\omega_2 \equiv 0  \quad \rm{mod}\ \theta
\end{gather*}
\end{Def}

Thus we have a sequence of classes of Monge--Amp{\`e}re systems:
\begin{gather*}
  \text{\{M-A system with Lagrangian pair\} \ $\supset$ \
\{M-A system with integrable Lag. pair\}}\\
\text{$\supset$ \ \{Hesse M-A system\} \ $\supset$ \ \{Euler--Lagrange
M-A system\} \ $\supset$ \ \{f\/lat M-A system\}}.
\end{gather*}

We will study Hesse Monge--Amp{\`e}re systems in detail,
and, in fact, we characterize the above three classes in intrinsic way.
Note that the equation of non-zero
constant Gauss--Kronecker curvature $K = c$ falls into the class of
Euler--Lagrange Monge--Amp{\`e}re systems and the equation of improper af\/f\/ine hyperspheres
${\rm Hess} = c$ falls into the class of f\/lat Monge--Amp{\`e}re systems.

We will show, in Sections~\ref{Homogeneous M-A systems with Lagrangian pairs.}.2--\ref{Homogeneous M-A systems with Lagrangian pairs.}.5,
that Monge--Amp{\`e}re systems def\/ined by
the Gaussian curvature constant equation $K=c$
in ${\mathbf{E}}^{n+1}$, $S^{n+1}$, $H^{n+1}$ are Euler--Lagrange systems.

\begin{Prop}
\label{Hesse-MA}
Let $\mathcal{M}$ be a Hesse Monge--Amp{\`e}re system, i.e.,
a Monge--Amp{\`e}re system
with a flat Lagrangian pair $(E_1,E_2)$ generated by
an $n$-form $\omega = \omega_1-\omega_2$ enjoying
the bi-decomposing condition.
Then $\mathcal{M}$ is locally isomorphic to
a Monge--Amp{\`e}re system $\mathcal{M}'$ on an open subset $U \subset {\mathbf{R}}^{2n+1}$
with the standard Lagrangian pair $(E^{\rm{st}}_1, E^{\rm{st}}_2)$
which is locally generated by an $n$-form of type
\begin{gather*}
F(x_1, \dots, x_n, z, p_1, \dots, p_n) dx_1\wedge\cdots\wedge dx_n - dp_1\wedge\cdots\wedge dp_n
\end{gather*}
for a non-vanishing function $F$ on $U$. In particular,
there exists a system of local Darboux coordinates
$(x_1, \dots, x_n, z, p_1, \dots, p_n)$
in some neighborhood of each point, such that
$\mathcal{M}$ is represented by
a Hesse Monge--Amp{\`e}re equation of the form
\begin{gather*}
\mathrm{Hess}(z) = F(x_1, \dots, x_n, z, p_1, \dots, p_n)
\qquad (F\not=0).
\end{gather*}
\end{Prop}

\begin{proof}
Since $(E_1,E_2)$ is f\/lat, around each point of~$M$,
there exists a system
of local coordinates
$
(x_1,\dots,x_n,z,p_1,\dots,p_n)
$
such that
$D =\{ \theta = 0 \}$, $\theta=dz-\sum\limits_{i=1}^np_idx_i$, and that
\begin{gather*}
E_1 =\mathrm{Ker}({\pi}_{2*})=\left\langle \frac{\partial}{\partial x_1}+
p_1\frac{\partial}{\partial z},
\dots,\frac{\partial}{\partial x_n}+p_n\frac{\partial}{\partial z}\right\rangle,
\\
E_2 =\mathrm{Ker}({\pi}_{1*})=\left\langle \frac{\partial}{\partial p_1},\dots,
\frac{\partial}{\partial p_n}\right\rangle,
\\
{\pi}_1 \colon  \  {\mathbf{R}}^{2n+1} \to {\mathbf{R}}^{n+1}, \qquad \pi_1(x_1, \dots, x_n, z,
p_1, \dots, p_n)=(x_1, \dots, x_n, z),
\\
{\pi}_2 \colon   \ {\mathbf{R}}^{2n+1} \to {\mathbf{R}}^{n+1},
\qquad
{\pi}_2(x_1, \dots, x_n, z, p_1, \dots, p_n) =
\left(p_1, \dots, p_n, \sum_{i=1}^np_ix_i-z\right).
\end{gather*}
This means that $\mathcal{M}$ is locally isomorphic to
a Monge--Amp{\`e}re system with the standard Lag\-rangian pair.
Then, from the bi-decomposing condition,
we have, setting $x = (x_1, \dots, x_n)$, $p = (p_1, \dots, p_n)$,
\begin{gather*}
\omega  = \omega_1-\omega_2
 =f(x,z,p)dx_1\wedge \cdots \wedge dx_n -
g(x,z,p)dp_1\wedge \cdots \wedge dp_n
\end{gather*}
for some functions $f(x,z,p)$, $g(x,z,p)$ $(\not=0)$ on
${\mathbf{R}}^{2n+1}$.
Therefore, putting $F=f/g$,
we have a~form $\mathrm{Hess}(z) = F(x,z,p)$.
\end{proof}

\begin{Prop}
\label{E-L-MA}
Let $\mathcal{M}$ be a Hesse Monge--Amp{\`e}re system.
Then $\mathcal{M}$ is an Euler--Lagrange Monge--Amp{\`e}re system
if and only if
$\mathcal{M}$ is locally isomorphic to a Monge--Amp{\`e}re system $\mathcal{M}'$
on an open subset $U \subset {\mathbf{R}}^{2n+1}$ with the standard Lagrangian pair generated by an $n$-form
of type
\begin{gather*}
F_1(x_1, \dots, x_n, z)\cdot F_2(\widetilde{z}, p_1, \dots, p_n)
dx_1\wedge\cdots\wedge dx_n - dp_1\wedge\cdots\wedge dp_n
\end{gather*}
for some non-vanishing functions $F_1$ of~$x$,~$z$ and $F_2$ of $\widetilde{z} = \sum\limits_{j=1}^n p_jx_j-z$
and~$p$.
In particular, ${\mathcal M}$~is locally represented as
\begin{gather*}
\mathrm{Hess}(z) =
F_1(x_1, \dots, x_n, z)\cdot F_2(\widetilde{z}, p_1, \dots, p_n)
\qquad
(F_1,\ F_2\not=0).
\end{gather*}
\end{Prop}

\begin{proof}
By the assumption, we can set
\begin{gather*}
\omega_1 = f(x,z,p)dx_1\wedge \cdots \wedge dx_n,
\qquad
\omega_2 =
g(x,z,p)dp_1\wedge \cdots \wedge dp_n.
\end{gather*}
Since
\begin{gather*}
d\omega_1   =df\wedge dx_1\wedge \cdots \wedge dx_n
=\left(\sum_{i=1}^n\frac{\partial f}{\partial x_i}dx_i+\frac{\partial f}{\partial z}dz+
\sum_{i=1}^n\frac{\partial f}{\partial p_i}dp_i\right)\wedge dx_1\wedge \cdots \wedge dx_n \\
\hphantom{d\omega_1}{} \equiv \sum_{i=1}^n\frac{\partial f}{\partial p_i} dp_i\wedge dx_1\wedge \cdots
\wedge dx_n,  \quad \rm{mod}\ \theta,
\end{gather*}
we see $d\omega_1 \equiv 0 \ \rm{mod}\ \theta$ if and only if
$f$ is independent of $p_1, \dots, p_n$:
$f=f(x_1, \dots, x_n,z)$.
Besides, we take another system of local coordinates $(x, \tilde{z}, p)$
with $\tilde{z} = \sum\limits_{i=1}^n x_ip_i - z$. Then
$\theta = - \Big(d\tilde{z} - \sum\limits_{i=1}^n x_idp_i\Big)$ and we have
\begin{gather*}
d\omega_2   =
dg\wedge dp_1\wedge \cdots \wedge dp_n
 =
\left(\sum_{i=1}^n\frac{\partial g}{\partial x_i}dx_i +\frac{\partial g}{\partial \tilde{z}}d\tilde{z}+
\sum_{i=1}^n\frac{\partial g}{\partial p_i}dp_i\right)
\wedge dp_1\wedge \cdots \wedge dp_n \\
\hphantom{d\omega_2}{}
 \equiv
\sum_{i=1}^n\frac{\partial g}{\partial x_i}
dx_i \wedge dp_1\wedge \cdots \wedge dp_n, \quad \rm{mod}\ \theta.
\end{gather*}
Therefore we see $d\omega_2 \equiv 0 \ \rm{mod}\ \theta$ if and only if~$g$ is independent of~$x$ for the system of local coordinates
$(x, \tilde{z}, p)$:
$g = g(\widetilde{z}, p_1, \dots, p_n)$.
Therefore putting $F_1 = f$, $F_2 = 1/g$,
we have a form
$\mathrm{Hess}(z) = F_1(x, z)\cdot F_2(\widetilde{z}, p)$.
\end{proof}

By Theorem \ref{Uniqueness}, we have the following:

\begin{Prop}
\label{contact invariance}
Let $n \geq 3$. Then Definition~{\rm \ref{Hesse-M-A-def}} $($resp.\ Definition~{\rm \ref{E-L-M-A-def})}
depends only on the Monge--Amp{\`e}re system and
does not depend on the choice of Lagrangian pairs of the Monge--Amp{\`e}re system.
The class of Hesse Monge--Amp{\`e}re systems $($resp.\
the class of Euler--Lagrange Monge--Amp{\`e}re systems$)$
is invariant under contact transformations.
\end{Prop}

\begin{proof}
Let ${\mathcal M}$ be a Monge--Amp{\`e}re system with a Lagrangian pair on a contact manifold
$(M, D)$.
Let $(E_1, E_2)$ and $(E_1', E_2')$ be two Lagrangian pairs associated to ${\mathcal M}$.
Then, by Theorem~\ref{Uniqueness}, $E_1' = E_1$, $E_2' = E_2$ or $E_1' = E_2$, $E_2' = E_1$.
Therefore the f\/latness of Lagrangian pair depends only on~${\mathcal M}$.
Moreover it is clear that, the condition of Def\/inition~\ref{E-L-M-A-def},
i.e., the possibility of a bi-decomposition $\omega =
\omega_1 - \omega_2$ of a local generator~$\omega$ into closed decomposable
forms $\omega_1$, $\omega_2$ up to a contact form~$\theta$ depends only on~${\mathcal M}$.

Let $\Phi \colon  (M, D) \to (M', D')$ be a contact transformation between~$(M, D)$ and
another contact manifold $(M', D')$. Set ${\mathcal M}' = {\Phi^{-1}}^*{\mathcal M}$.
Then ${\mathcal M}'$ is a Monge--Amp{\`e}re system with the Lagrangian pair
$(\Phi_*E_1, \Phi_*E_2)$. If $(E_1, E_2)$ is f\/lat, then so is $(\Phi_*E_1, \Phi_*E_2)$.
Moreover if $\omega = \omega_1 - \omega_2$ is a bi-decomposition satisfying
the condition of Def\/inition~\ref{E-L-M-A-def} for ${\mathcal M}$, then
$({\Phi^{-1}})^*\omega = ({\Phi^{-1}})^*\omega_1 - ({\Phi^{-1}})^*\omega_2$
satisf\/ies the condition of Def\/inition \ref{E-L-M-A-def} for~${\mathcal M}'$.
\end{proof}

\begin{Rem}
\label{Poincare-Cartan form}
The Euler--Lagrange systems are studied in \cite{B-G-G} via the key notion ``Poincar\'{e}--Cartan form''.
If an Euler--Lagrange Monge--Amp{\`e}re system with Lagrangian pair
is given by
\begin{gather*}
\theta = dz - \sum_{i=1}^n p_i dx_i = - \left(d\tilde{z} - \sum_{i=1}^n x_i dp_i\right), \\
\omega = f(x, z) dx_1 \wedge \cdots \wedge dx_n - g(p, \tilde{z}) dp_1\wedge \cdots \wedge dp_n,
\end{gather*}
then the Poincar\'{e}--Cartan form is given by
\begin{gather*}
\Pi = \theta\wedge\omega = f(x, z) dz \wedge dx_1 \wedge \cdots \wedge dx_n
- g(p, \tilde{z}) d\tilde{z} \wedge dp_1\wedge \cdots \wedge dp_n.
\end{gather*}
\end{Rem}

To conclude this section,
we characterize, among others, the class of equations $\mathrm{Hess} = c$
in term of the projective structure:

\begin{Prop}
Let $\mathcal{M}$ be an Euler--Lagrange Monge--Amp{\`e}re system
on $M=P(T^*P^{n+1})$ induced by the diagram
\begin{gather*}
\xymatrix{
& P\big(T^*P^{n+1}\big) \ar[ld]_{\pi_1} \ar[rd]^{\pi_2}  &\\
P^{n+1} &  & {P^{n+1}}^*}
\end{gather*}
and $W_1=P^{n+1}$, $W_2={P^{n+1}}^* $,
generated by a bi-decomposable
$n$-form $\omega = \omega_1-\omega_2$.
Then the condition
\begin{gather*}
\nabla \omega_1=0, \quad \nabla \omega_2=0
\end{gather*}
is satisfied for the covariant derivative $\nabla$ of
the flat connection induced on each local projective chart
of~$W_1=P^{n+1}$
if and only if~${\mathcal M}$ is represented by
a Monge--Amp{\`e}re equation
\begin{gather*}
\mathrm{Hess}(z) = c \quad (c \ \text{\rm is constant}, \ c\not=0).
\end{gather*}
\end{Prop}

\begin{proof}
From the equivalence between the conditions
$\nabla \omega_1=0$ and $df=0$
(resp.\ $\nabla \omega_2=0$ and $dg=0$) in the proof of
Proposition~\ref{Hesse-MA},
we have that $f$, $g$ are non-zero constants.
Therefore we have the form $\mathrm{Hess}(z) = c$ $(\not=0)$.
\end{proof}

\section[A method to construct Monge-Amp{\`e}re systems
with Lagrangian pairs]{A method to construct Monge--Amp{\`e}re systems\\
with Lagrangian pairs}
\label{A method to construct systems with Lagrangian pairs.}

Let $(M, D)$ be a contact manifold of dimension $2n+1$ and
$(E_1, E_2)$ a Lagrangian pair on $(M, D)$.
Consider the quotient bundle $TM/E_2$ (resp.~$TM/E_1$) of rank~$n+1$
and a~section~$\omega'_1$ (resp.~$\omega'_2$)
to the line bundle
$\wedge^{n+1}(TM/E_2)^*$ (resp.~$\wedge^{n+1}(TM/E_1)^*$) of\/f
the zero-section.
Let $\theta$ be a local contact form def\/ining $D$.
Recall that the Reeb vector f\/ield $R = R_{\theta}$ is def\/ined by the condition
$i_{R_\theta} \theta = 1$, $i_{R_\theta} d\theta = 0$.
Then we def\/ine an $n$-form on $M$ by
\begin{gather*}
\omega = i_{R_\theta} \big(\Pi_1^*\omega'_1 - \Pi_2^*\omega'_2\big).
\end{gather*}
Here $\Pi_1\colon TM \to TM/E_2$ (resp.\ $\Pi_2\colon TM \to TM/E_1$)
denotes the bundle projection, and
$\Pi_1^*\colon (TM/E_2)^* \to T^*M$ and $\Pi_1^*\colon \wedge^{n+1}(TM/E_2)^* \to
\wedge^{n+1}T^*M$
(resp.\ $\Pi_2^*\colon (TM/E_1)^* \to T^*M$ and $\Pi_2^*\colon
\wedge^{n+1}(TM/E_1)^* \to \wedge^{n+1}T^*M$)
its dual injections.
Then we have the following basic lemma for our construction:

\begin{Lem}
\label{MA-system construction}
The differential system $\mathcal{M} = \langle \theta, d\theta,
\omega \rangle$, generated by $\omega$ and the contact form $\theta$,
is independent of the choice of $\theta$,
and depends only on given $\omega'_1$, $\omega'_2$.
\end{Lem}

We call $\mathcal{M}$ {\it the Monge--Amp{\`e}re system with} $(E_1, E_2)$
{\it induced from} $\omega'_1$, $\omega'_2$.

To show Lemma \ref{MA-system construction},
take a local symplectic frame $X_1, \dots, X_n, P_1, \dots, P_n$
of $D$ with respect to $d\theta$:
\begin{gather*}
d\theta(X_i, X_j) = 0, \qquad d\theta(P_i, P_j) = 0, \qquad
d\theta(X_i, P_j) = \delta_{ij},
\end{gather*}
with
\begin{gather*}
E_1 = \langle X_1, \dots, X_n\rangle, \qquad
E_2 = \langle P_1, \dots, P_n\rangle.
\end{gather*}
Then $X_1, \dots, X_n, P_1, \dots, P_n, R_\theta$ form a local frame of~$TM$.

\begin{Lem}
\label{Reeb}
The Reeb vector field $R_{\theta'}$ for a contact form
$\theta' = \rho\theta$ defining $D$ is given by
\begin{gather*}
R_{\theta'} = \frac{1}{\rho^2}\left[
\sum_{i=1}^n (P_i\rho)X_i - \sum_{i=1}^n (X_i\rho)P_i + \rho R_{\theta}
\right].
\end{gather*}
\end{Lem}

\begin{proof}
Let
$\alpha_1, \dots, \alpha_n, \beta_1, \dots, \beta_n, \theta$ be
the frame of $T^*M$ dual to
$X_1, \dots, X_n, P_1, \dots, P_n, R_\theta$.
Set $R_{\theta'} = \sum_i a_iX_i + \sum_j b_jP_j + cR_{\theta}$.
Then, by $i_{R_{\theta'}} \theta' = 1$, we have $c = \frac{1}{\rho}$.
Besides we have
\begin{gather*}
i_{R_{\theta'}} d\theta'  =  [i_{R_{\theta'}} d\rho]\theta -
i_{R_{\theta'}} \theta\ d\rho
+ \rho i_{R_{\theta'}} d\theta
\\
\hphantom{i_{R_{\theta'}} d\theta'}{}
 =
\left[\sum_i a_i(X_i\rho) + \sum_j b_j(P_j\rho) + c(R_\theta\rho)\right]\theta
- c d\rho + \rho\left[\sum_i a_i\beta_i - \sum_j b_j\alpha_j\right],
\end{gather*}
while $d\rho = \sum_i (X_i\rho)\alpha_i + \sum_j (P_j\rho)\beta_j
+ (R_\theta \rho)\theta$.
Therefore, by $i_{R_{\theta'}} d\theta' = 0$, we have
\begin{gather*}
- \sum_i [\rho b_i + c(X_i\rho)]\alpha_i +
\sum_j [\rho a_j - c(P_j\rho)]\beta_j
+ \left[\sum_i a_i(X_i\rho) + \sum_j b_j(P_j\rho)\right]\theta = 0.
\end{gather*}
Thus we have
\begin{gather*}
a_i = \frac{1}{\rho^2} P_i\rho, \qquad
b_i = - \frac{1}{\rho^2} X_i\rho, \qquad
c = \frac{1}{\rho}.\tag*{\qed}
\end{gather*}
\renewcommand{\qed}{}
\end{proof}

\begin{proof}[Proof of Lemma \ref{MA-system construction}]
We set $\Pi_1^*{\omega'}_1 =
\lambda\cdot \theta\wedge\alpha_1\wedge\cdots\wedge\alpha_n$
for a function~$\lambda$.
Then
$i_{R_{\theta}} (\Pi_1^*{\omega'}_1)
= \lambda\cdot \alpha_1\wedge\cdots\wedge\alpha_n$.
By Lemma~\ref{Reeb}, we have for~$\theta'$
\begin{gather*}
i_{R_{\theta'}} \big(\Pi_1^*{\omega'}_1\big)
\equiv \frac{1}{\rho}\lambda\cdot \alpha_1\wedge\cdots\wedge\alpha_n
= \frac{1}{\rho}\cdot i_{R_{\theta}} {\omega'}_1,\quad \rm{mod}\ \theta.
\end{gather*}
Similarly we have
$i_{R_{\theta'}} (\Pi_2^*{\omega'}_2)
\equiv \frac{1}{\rho}\cdot i_{R_{\theta}} {\omega'}_2$, $\rm{mod}\ \theta$.
Therefore
\begin{gather*}
i_{R_{\theta'}} \big(\Pi_1^*{\omega'}_1 - \Pi_2^*{\omega'}_2\big) \equiv
\frac{1}{\rho} i_{R_{\theta}} \big(\Pi_1^*{\omega'}_1 - \Pi_2^*{\omega'}_2\big),\quad
\rm{mod}\ \theta.
\end{gather*}
Therefore $\mathcal{M} = \langle \theta, d\theta, \omega\rangle$
is independent of the choice of~$\theta$.
\end{proof}

Suppose that $(E_1, E_2)$ is integrable.
Let
\begin{gather*}
\xymatrix{
& M^{2n+1} \ar[ld]_{\pi_1} \ar[rd]^{\pi_2}  &\\
W^{n+1}_1 &  & W^{n+1}_2}
\end{gather*}
be a double Legendrian f\/ibration induced by a
Lagrangian pair $(E_1, E_2)$ locally.
Suppose that
a volume $(n+1)$-form ${\Omega_1}$ is given on $W_1$
(resp.\ a~volume $(n+1)$-form ${\Omega_2}$ on $W_2$).
Since
$(TM/E_2)_x \cong T_{\pi_1(x)}W_1$ via $\pi_1$
(resp.\ $(TM/E_1)_x \cong T_{\pi_2(x)}W_2$ via $\pi_2$) for $x\in M$,
${\Omega_1}$ (resp.~${\Omega_2}$)
is regarded as a non-zero section $\omega'_1$ of
$\wedge^{n+1}(TM/E_1)^*$ (resp.~$\omega'_2$ of
$\wedge^{n+1}(TM/E_2)^*$).
Then, following the general setting,  we set
\begin{gather*}
\omega = i_{R_\theta} \big(\pi_1^*\Omega_1 - \pi_2^*\Omega_2\big).
\end{gather*}
Note that $\pi_1^*\Omega_1$, $\pi_2^*\Omega_2$
are basic forms (cf.~\cite{I-L}) for $\pi_1$, $\pi_2$ respectively, however
$i_{R_\theta} \pi_1^*\Omega_1$ and $i_{R_\theta} \pi_2^*\Omega_2$
need not to be basic.

\begin{Ex}
\label{standard model calcu}
Consider ${\mathbf{R}}^{2n+1}$ with coordinates $(x, z, p) =
(x_1, \dots, x_n, z$, $p_1, \dots, p_n)$, $D = \{ \theta = 0\}$,
$\theta = dz - \sum\limits_{i=1}^n p_i dx_i$,
$E^{\rm{st}}_1 =
\langle \frac{\partial}{\partial x_1} + p_1\frac{\partial}{\partial z},
\dots, \frac{\partial}{\partial x_n} + p_n\frac{\partial}{\partial z}\rangle$,
$E^{\rm{st}}_2 =
\langle \frac{\partial}{\partial p_1},
\dots, \frac{\partial}{\partial p_n}\rangle$, $\pi_1 \colon  {\mathbf{R}}^{2n+1} \to
{\mathbf{R}}^{n+1}$, $\pi_1(x, z, p) = (x, z)$, $\pi_2 \colon  {\mathbf{R}}^{2n+1} \to
{\mathbf{R}}^{n+1}$, $\pi_2(x, z, p) = (p, x\cdot p - z)$. Set $\tilde{z} = x\cdot p - z$.
The Reeb vector f\/ield for $\theta$ is given by
$R = \frac{\partial}{\partial z}$.

Let $\pi_1^*\omega'_1 = f(x, z, p) dz \wedge dx_1\wedge \cdots \wedge dx_n$
be a non-zero section of $\wedge^{n+1}(T{\mathbf{R}}^{2n+1}/E_2)^*$
pulled-back to an $(n+1)$-form on ${\mathbf{R}}^{2n+1}$.
Then $i_R (\pi_1^*\omega'_1) = f(x, z, p) dx_1\wedge \cdots \wedge dx_n$.
Similarly, let
$\pi_2^*\omega'_2 = -g(x, z, p) d\tilde{z} \wedge dp_1\wedge \cdots \wedge
dp_n$ be a non-zero section of $\wedge^{n+1}(T{\mathbf{R}}^{2n+1}/E_1)^*$
pulled-back to an $(n+1)$-form on~${\mathbf{R}}^{2n+1}$.
Then $i_R (\pi_2^*\omega'_2) = g(x, z, p) dp_1\wedge \cdots \wedge dp_n$.
Thus, following the general setting, we have
\begin{gather*}
\omega = f(x, z, p)dx_1\wedge \cdots \wedge dx_n -
g(x, z, p) dp_1\wedge \cdots \wedge dp_n,
\end{gather*}
and we obtain Hesse Monge--Amp{\`e}re systems.

Further, let $\Omega_1 =
f(x, z) dz \wedge dx_1\wedge \cdots \wedge dx_n$
(resp.\ $\Omega_2 = - g(p, \tilde{z}) d\tilde{z} \wedge dp_1\wedge \cdots
\wedge dp_n$) be
a~volume form on~$W_1 = {\mathbf{R}}^{n+1}$ (resp.\ on~$W_2 = {\mathbf{R}}^{n+1}$).
Then $i_R \pi_1^*\Omega_1 =
f(x, z) dx_1\wedge \cdots \wedge dx_n$
(resp.\ $i_R \pi_2^*\Omega_2 =
g(p, \tilde{z}) dp_1\wedge \cdots \wedge dp_n$).
Thus we have
\begin{gather*}
\omega = f(x, z) dx_1\wedge \cdots \wedge dx_n
- g(p, \tilde{z}) dp_1\wedge \cdots \wedge dp_n,
\end{gather*}
and we obtain Euler--Lagrange Monge--Amp\`ere systems.
\end{Ex}

\begin{Rem}
The Poincar\'{e}--Cartan form is given by $\Pi = \pi_1^*\Omega_1 - \pi_2^*\Omega_2$ (see Remark~\ref{Poincare-Cartan form}).
\end{Rem}

\section{Homogeneous Monge--Amp{\`e}re systems with Lagrangian pairs}
\label{Homogeneous M-A systems with Lagrangian pairs.}

{\bf 8.1.}
Monge--Amp{\`e}re system on ${\mathbf{R}}^{2n+1}$ as
${\rm Hess} = c$ on~${\mathbf{R}}^{n+1}$.
The Monge--Amp{\`e}re system with Lagrangian pair
for the equation ${\rm Hess}(f)= c$
in equi-af\/f\/ine geometry is given as follows.

Consider the contact manifold
$M = {\mathbf{R}}^{2n+1}$ with coordinates
\begin{gather*}
(x, z, p) = (x_1, \dots, x_n, z, p_1, \dots, p_n)
\end{gather*}
and with the contact form
$\theta = dz - \sum\limits_{i=1}^n p_idx_i$.
We set two Lagrangian sub-bundles $E_1$, $E_2$
of the contact distribution $D=\{ \theta =0 \}$ by
\begin{gather*}
E_1  =
\left\langle \frac{\partial}{\partial x_1}+p_1\frac{\partial}{\partial z},
\dots,\frac{\partial}{\partial x_n}+p_n\frac{\partial}{\partial z}\right\rangle,
\qquad
E_2 =
\left\langle \frac{\partial}{\partial p_1},\dots,\frac{\partial}{\partial p_n}
\right\rangle,
\end{gather*}
which form a Lagrangian pair of $(M, D)$.
The Reeb vector f\/ield~$R$ is given by~$\frac{\partial}{\partial z}$.
Note that
$
- \theta = d\Big(\sum\limits_{i=1}^n p_ix_i - z\Big) - \sum\limits_{i=1}^n x_idp_i.
$
Then we have the double Legendrian f\/ibration induced by~$(E_1,E_2)$:
\begin{gather*}
\xymatrix{
& M={\mathbf{R}}^{2n+1} \ar[ld]_{\pi_1} \ar[rd]^{\pi_2}  &\\
W_1={\mathbf{R}}^{n+1} &  & W_2={\mathbf{R}}^{n+1}}
\end{gather*}
where
\begin{gather*}
\pi_1(x, z, p) = (x, z), \qquad
\pi_2(x, z, p) =
(p, \tilde{z}), \qquad \tilde{z} = \sum_{i=1}^n p_ix_i - z.
\end{gather*}
Moreover take the
$(n+1)$-forms
\begin{gather*}
\Omega_1
= c(dz \wedge dx_1\wedge \cdots \wedge dx_{n}),\qquad
\Omega_2
= - d\tilde{z} \wedge dp_1\wedge \cdots \wedge dp_{n}
\end{gather*}
on $W_1 = {\mathbf{R}}^{n+1}$
($c\in {\mathbf{R}}$, $c\not= 0$)
and
on $W_2 = {\mathbf{R}}^{n+1}$ respectively.
Then $\omega = i_R (\pi_1^*\Omega_1 - \pi_2^*\Omega_2)$
is given by
\begin{gather*}
\omega = c dx_1\wedge \cdots \wedge dx_{n} -
dp_1\wedge \cdots \wedge dp_{n}.
\end{gather*}

Thus we construct
the Monge--Amp{\`e}re system $\mathcal{M}=\langle \theta, d\theta,
\omega \rangle $
with the Lagrangian pair $(E_1,E_2)$ globally on $M={\mathbf{R}}^{2n+1}$.

\begin{Prop}
Under the situation above,
we have a Monge--Amp{\`e}re system
$\mathcal{M}$ generated by $(\theta, d\theta, \omega)$
with a Lagrangian pair $(E_1,E_2)$ on $M={\mathbf R}^{2n+1}$.
The projection to $W_1 = {\mathbf{R}}^{n+1}$
of a~geometric solution of $\mathcal{M}$ satisfies the equation
$\textrm{Hess}(f)=c$, when it is represented as a~graph
$z = f(x_1,\dots,x_n)$
outside of its singular locus.
The projection to $W_2 = {\mathbf{R}}^{n+1}$
of a~geometric solution of $\mathcal{M}$ satisfies the equation
$\textrm{Hess}(f)=\frac{1}{c}$, when it is represented as a~graph
$z'=\sum\limits_{i=1}^np_ix_i-z = f(p_1,\dots,p_n)$
outside of its singular locus.
\end{Prop}

\begin{Rem}
The Monge--Amp{\`e}re system for ${\mathrm{Hess}} = c$
is isomorphic to the system for ${\mathrm{Hess}} = 1$ if $c > 0$, and
is isomorphic to the system for ${\mathrm{Hess}} = -1$ if $c < 0$.
\end{Rem}

\begin{Rem}
The Monge--Amp{\`e}re system for ${\mathrm{Hess}} = c, c \not= 0$ has the natural symmetry
by the group $G'$ of equi-af\/f\/ine transformations on $W_1={\mathbf{R}}^{n+1}$
preserving the vector f\/ield $\frac{\partial}{\partial z}$.
The group $G'$ is given by the semi-direct product $G' = G''\ltimes {{\mathbf{R}}}^{n+1}$
of $G'' \subset {\rm SL}(n+1, {\mathbf{R}})$
and ${\mathbf{R}}^{n+1}$, where
\begin{gather*}
G'' =\left\{ \left.\left(
\begin{matrix}
A & 0 \\
{}^t\!a & 1
\end{matrix}
\right)\,
\right\vert  \,
A \in  {\rm SL}(n, {\mathbf{R}}), \,a \in {\mathbf{R}}^n \right\}.
\end{gather*}
Note that $\dim G' = n(n+2)$ and also that each element of $G'$
is identif\/ied with
\begin{gather*}
\left(
\begin{matrix}
1 & 0 & 0 \\
b & A & 0 \\
c & {}^t\!a & 1
\end{matrix}
\right)
\end{gather*}
via an appropriate embedding ${\rm SL}(n+1, {\mathbf{R}}) \hookrightarrow {\rm GL}(n+2, {\mathbf{R}})$.
The Monge--Amp{\`e}re system for ${\mathrm{Hess}} = c, c \not= 0$,
in fact,
has bigger symmetry which attains the maximum for the dimension estimate
of automorphisms given in Section~\ref{Automorphisms}
for bi-decomposable Monge--Amp{\`e}re systems.

Let $G$ be a subgroup of the projective transformation group ${\rm PGL}(n+2, {\mathbf{R}})$
on ${\mathbf{R}}^{n+2}$ consisting of transformations represented by the matrices
\begin{gather*}
\widetilde{A} = \left(
\begin{matrix}
\ell & 0 & 0 \\
b & A & 0 \\
c & {}^t\!a & k
\end{matrix}
\right),
\end{gather*}
considered up to non-zero scalar multiples.
Here $\ell, k \in {\mathbf{R}}^{\times}$, $A \in {\rm GL}(n, {\mathbf{R}})$, $a, b \in {\mathbf{R}}^n, c \in {\mathbf{R}}$
satisfying the condition $(\det A)^2 = (k\ell)^n$.
From the condition, we can take $\det A = \pm 1$, $\ell = \pm 1/k$ in ${\rm PGL}(n+2, {\mathbf{R}})$.
If $n$ is odd, then we can take $\det A = 1$ and $\ell = 1/k$.
Note that $\dim G = (n+1)^2$.

The group $G$ acts on $M = {\mathbf{R}}^{2n+1} = {\mathbf{R}}^n \times {\mathbf{R}} \times {\mathbf{R}}^n$ transitively by
\begin{gather*}
\widetilde{A}(x, z, p) = \left(\frac{1}{\ell}(Ax + b), \, \frac{1}{\ell}\big(kz - {}^t\!ax - c\big), \, k\left(p - \frac{1}{k}{}^t\!a\right)A^{-1}\right).
\end{gather*}
Here $a$, $b$, $x$ are regarded as column $n$-vectors and $p$ a row $n$-vector.
The contact form $\theta = dz - pdx$ is transformed to
$\frac{k}{\ell}\theta$, therefore
the contact structure $D=\{ \theta =0 \}$ is $G$-invariant.
The bi-decomposable form $\omega$ is transformed to
\begin{gather*}
\frac{1}{\ell^n}(\det A)c(dx_1\wedge \cdots \wedge dx_{n})
-
k^n(\det A)^{-1}(dp_1\wedge \cdots \wedge dp_{n}) = \frac{1}{\ell^n}(\det A)  \omega,
\end{gather*}
by the condition $(\det A)^2 = (k\ell)^n$.
Therefore $G$ leaves the Monge--Amp{\`e}re system ${\mathcal M} = \langle \theta, d\theta, \omega\rangle$
for ${\mathrm{Hess}} = c$, $c \not= 0$.

The group $G$ acts on $W_1 = {\mathbf{R}}^{n+1}$ and on $W_2={\mathbf{R}}^{n+1}$ respectively by
\begin{gather*}
\widetilde{A}(x, z) = \left(\frac{1}{\ell}(Ax + b), \, \frac{1}{\ell}\big(kz - {}^t\!ax - c\big)\right),
\end{gather*}
and by
\begin{gather*}
\widetilde{A}(\tilde{z}, p) = \left(\frac{1}{\ell}\left(k\tilde{z} + k\left(p - \frac{1}{k}{}^t\!a\right)A^{-1}b + c\right), \, k\left(p - \frac{1}{k}{}^t\!a\right)A^{-1}\right).
\end{gather*}
Then the volume form $\Omega_1$ on~$W_1$ (resp.~$\Omega_2$ on~$W_2$)
is transformed to $\frac{k}{\ell^{n+1}}(\det A)\Omega_1$ (resp.\ $\frac{k^{n+1}}{\ell}(\det A)^{-1}\Omega_2$).
\end{Rem}

{\bf 8.2.}
Monge--Amp{\`e}re system on $T_1{\mathbf{E}}^{n+1}$
as $K=c$ in ${\mathbf{E}}^{n+1}$.
In Sections~8.2--8.4, we describe
the Monge--Amp{\`e}re systems corresponding to the
the equation of constant Gaussian curvature in Euclidean, spherical or hyperbolic geometry.
To provide the concrete form of the Monge--Amp{\`e}re system, we treat three cases separately.

In the famous paper of Gauss~\cite{S},
the ``Gaussian curvature'' of a space surface
is introduced
as the ratio of areas in the ``Gauss map'' of the surface.
We observe that the equation of constant Gaussian curvature is
regarded as a Monge--Amp{\`e}re system with Lagrangian pair as follows.

Consider the unit tangent bundle $T_1{\mathbf{E}}^{n+1}={\mathbf{E}}^{n+1}\times S^n$
of the Euclidean space ${\mathbf{E}}^{n+1}$.
The standard contact structure on ${\mathbf{E}}^{n+1}\times S^n$
is
given by the one-form $\theta
= y_1dx_1 + y_2dx_2 + \cdots + y_{n+1}dx_{n+1}$ on
${\mathbf{E}}^{n+1}\times {\mathbf{E}}^{n+1}$,
restricted to ${\mathbf{E}}^{n+1}\times S^n$.
Here
\begin{gather*}
(x;y)=(x_1,x_2,\dots,x_{n+1};y_1,y_2,\dots,y_{n+1})
\end{gather*}
is the system of coordinates on ${\mathbf{E}}^{n+1}\times {\mathbf{E}}^{n+1}$.
We set the contact distribution $D=\{ \theta = 0 \} \subset
T({\mathbf{E}}^{n+1}\times S^n)$ and two Lagrangian subbundles of $D$:
\begin{gather*}
E_1
=\left\{ u = \xi_1 \frac{\partial}{\partial x_1}+
\xi_2 \frac{\partial}{\partial x_2}+\cdots +
\xi_{n+1}\frac{\partial}{\partial x_{n+1}}
\,\Big|\, \xi_1y_1+\xi_2y_2+\cdots +\xi_{n+1}y_{n+1}=0 \right\},
\\
E_2
=\left\{ v=\eta_1 \frac{\partial}{\partial y_1}+
\eta_2 \frac{\partial}{\partial y_2}+\cdots +
\eta_{n+1}\frac{\partial}{\partial y_{n+1}}
\,\Big|\, v\ \text{is tangent to}\ S^n \right\},
\end{gather*}
which form an integrable Lagrangian pair of~$(M, D)$.
Then we have the double Legendrian f\/ibration induced by $(E_1,E_2)$:
\begin{gather*}
\xymatrix{
& M={\mathbf{E}}^{n+1}\times S^n \ar[ld]_{\pi_1} \ar[rd]^{\pi_2}  &\\
W_1={\mathbf{E}}^{n+1} &  & W_2={\mathbf{R}}\times S^n}
\end{gather*}
Here $\pi_1(x,y)=x, \pi_2(x,y)=(x\cdot y,y)$ for
$(x,y)\in {\mathbf{E}}^{n+1}\times S^n \subset {\mathbf{E}}^{n+1} \times {\mathbf{E}}^{n+1}$.

\begin{Lem}
\label{Euclid-M-A}
The above Lagrangian pair is f\/lat, namely, the bi-Legendrian f\/ibration is
contactomorphic to the standard one: The Lagrangian pair $(E_1, E_2) = ({\rm Ker}(\pi_2)_*, {\rm Ker}(\pi_1)_*)$
is flat
$($see the standard example in Introduction$)$.
\end{Lem}

\begin{proof}
On the open sets $U = \{ (x, y) \in {\mathbf{E}}^{n+1}\times S^n
\,|\, y_{n+1} \not= 0 \}$ of ${\mathbf{E}}^{n+1}\times S^n$
and $V = \{ (w, y) \in {\mathbf{R}}\times S^n \,|\,
 y_{n+1} \not= 0 \}$ of ${\mathbf{R}}\times S^n$,
take the system of coordinates $x'_i = x_i$, $p'_i = - \frac{y_i}{y_{n+1}}$,
$z' = x_{n+1}$, $1 \leq i \leq n$, on~$U$, and def\/ine
the dif\/feomorphism $\Phi \colon  U \to {\mathbf{R}}^{2n+1}$ by $\Phi(x, y) = (x', z', p')$.
Moreover take the system of coordinates
$\tilde{z}' = \frac{w}{y_{n+1}}$, $p'_i = - \frac{y_i}{y_{n+1}}$,
$1 \leq i \leq n$, on $V$ and def\/ine the dif\/feomorphism
$\psi \colon  V \to {\mathbf{R}}^{n+1}$ by $\psi(w, y) = (\tilde{z}', p')$.
We denote by $\varphi \colon  {\mathbf{E}}^{n+1} \to
{\mathbf{R}}^{n+1}$ the identity map, forgetting the Euclidean metric.
Then $(\Phi, \varphi, \psi)$ induces the contactomorphism from
$(U, D; E_1, E_2)$ to $({\mathbf{R}}^{2n+1}, D_{\rm{st}};
E^{\rm{st}}_1, E^{\rm{st}}_2)$.
In fact $\theta = y_{n+1}\Big(dz' - \sum\limits_{i=1}^n p'_i dx'_i\Big) =
y_{n+1}(\Phi^{-1})^*\theta_{\rm{st}}$ and
$y_{n+1} = \pm 1\Big/\sqrt{1 + \sum\limits_{i=1}^n (p'_i)^2}$.
Similarly on each open set $\{ y_i \not= 0\}$, $1 \leq i \leq n+1$,
we see the f\/latness of $(E_1, E_2)$.
\end{proof}

We endow ${\mathbf{E}}^{n+1}$
with the standard volume form
\begin{gather*}
\Omega_1 = c dx_1\wedge dx_2 \wedge \cdots \wedge dx_{n+1},
\end{gather*}
multiplied with a real constant~$c$ $(\not= 0)$.
Moreover we endow $(z;y_1,y_2,\dots,y_{n+1}) \in
{\mathbf{R}}\times S^n$ $(\subset {\mathbf{R}} \times {\mathbf{E}}^{n+1})$
with the standard volume form on ${\mathbf{R}}\times S^n$
\begin{gather*}
\Omega_2 = dz \wedge
\sum_{i=1}^{n+1}\big((-1)^{i+1}y_idy_1\wedge \cdots \wedge \breve{dy_i} \wedge
\cdots \wedge dy_{n+1}\big)\vert_{{\mathbf{R}}\times S^n}.
\end{gather*}
The Reeb vector f\/ield $R$ on ${\mathbf{E}}^{n+1}\times S^n$ is
given by $R = y_1\frac{\partial}{\partial x_1} + y_2\frac{\partial}{\partial x_2} +
\cdots + y_{n+1}\frac{\partial}{\partial x_{n+1}}$, the ``tautological'' vector f\/ield.
Then we
set $\omega = i_R (\pi_1^*\Omega_1 - \pi_2^*\Omega_2)$.
Since
\begin{gather*}
i_R \pi_2^*\Omega_2    =
 i_R d(x\cdot y)\wedge
\sum_{i=1}^{n+1}\big((-1)^{i+1}y_idy_1\wedge \cdots \wedge \breve{dy_i} \wedge
\cdots \wedge dy_{n+1}\big)\\
\hphantom{i_R \pi_2^*\Omega_2 }{}
   =
\sum_{i=1}^{n+1}((-1)^{i+1}y_idy_1\wedge \cdots \wedge \breve{dy_i} \wedge
\cdots \wedge dy_{n+1}),
\end{gather*}
we have
\begin{gather*}
\omega
   =
c(y_1dx_2\wedge \cdots \wedge dx_{n+1} - y_2dx_1\wedge dx_3\wedge \cdots
\wedge dx_{n+1} + \cdots + (-1)^ny_{n+1}dx_1\wedge \cdots \wedge dx_n) \\
 \hphantom{\omega=}{} -
(y_1dy_2 \wedge \cdots \wedge dy_{n+1}
- y_2dy_1 \wedge dy_3\wedge \cdots \wedge dy_{n+1} + \cdots +
 (-1)^ny_{n+1}dy_1 \wedge \cdots \wedge dy_n)
\end{gather*}
on ${\mathbf{E}}^{n+1}\times S^n$.

This is exactly the reincarnation of the equation $K = c$
from the original def\/inition due to Gauss.

\begin{Prop}
Under the situation above,
we have a Monge--Amp{\`e}re system
$\mathcal{M}$ generated by $(\theta, d\theta, \omega)$
with a Lagrangian pair $(E_1,E_2)$ on $M=T_1{\mathbf{E}}^{n+1}={\mathbf{E}}^{n+1}\times S^n$.
The Gaussian curvature
of the projection to $W_1={\mathbf{E}}^{n+1}$ of a geometric solution
of~$\mathcal{M}$
is equal to constant~$c$ outside of singular locus.
\end{Prop}

\begin{Rem}
The Monge--Amp{\`e}re system for $K = c$
is isomorphic to the system for $K = 1$ if $c > 0$,
and is isomorphic to the system for $K = -1$ if
$c < 0$.
\end{Rem}

\begin{Rem}
Let $G$
be the Euclidean group on ${\mathbf{E}}^{n+1}$.
The group $G$ acts on the orthonormal frame bundle transitively, therefore so
on the unit tangent sphere bundle
(the orthonormal $1$-Stiefel bundle)
${\mathbf{E}}^{n+1}\times S^n$ of ${\mathbf{E}}^{n+1}$.
We f\/ix the origin $0$ in ${\mathbf{E}}^{n+1}$ and identify
${\mathbf{E}}^{n+1}$ with ${\mathbf{R}}^{n+1}$, giving
the isomorphism $G \cong {\rm O}(n+1) \ltimes {\mathbf{R}}^{n+1}$.
The contact structure $D = \{ \theta = 0\}$ and
the contact form $\theta = \sum\limits_{i=1}^{n+1} y_idx_i$ on
${\mathbf{E}}^{n+1}\times S^n$ are $G$-invariant, and
any $G$-invariant contact form is a non-zero multiple of~$\theta$.

The group $G$ acts on ${\mathbf{R}} \times S^n$ by
\begin{gather*}
g(r, v) = (g(0)\cdot g(v) + r, g(v)) = (b\cdot Av + r,
Av),
\end{gather*}
Here, for $g \in G$, we set $g(v) = Av + b$
($A \in {\rm O}(n+1)$, $b \in {\mathbf{R}}^{n+1}$).
Then we get the diagram:
\begin{gather*}
{\mathbf{E}}^{n+1} = G/H' \leftarrow {\mathbf{E}}^{n+1}\times S^n = G/H
\rightarrow {\mathbf{R}}\times S^n = G/H'',
\end{gather*}
for the isotropy groups $H$, $H'$ and $H''$ satisfying
\begin{gather*}
H' \cong {\rm O}(n+1) \hookleftarrow H \cong {\rm O}(n)
\hookrightarrow H'' \cong {\rm O}(n) \ltimes {\mathbf{R}}^n.
\end{gather*}
The $G$-invariant volume forms on ${\mathbf{E}}^{n+1}$ and ${\mathbf{R}}\times S^n$
are unique up to non-zero constant.
We can construct
the Monge--Amp{\`e}re system $\mathcal{M}=\langle \theta, d\theta,
\omega \rangle $
with the Lagrangian pair $(E_1,E_2)$ globally on $M=T_1{\mathbf{E}}^{n+1}$.
\end{Rem}

{\bf 8.3.}
Monge--Amp{\`e}re system on $T_1S^{n+1}$
as $K=c$ in $S^{n+1}$.
Consider
${\mathbf{E}}^{n+2}\times {\mathbf{E}}^{n+2}$ with
coordinates
$x = (x_0, x_1, \dots, x_{n+1}), y = (y_0, y_1, \dots, y_{n+1})$.
Set $x\cdot y = \sum_{i=0}^{n+1}x_iy_i$,
the standard inner product.
Consider
\begin{gather*}
T_1S^{n+1} = \big\{ (x, y) \in  {\mathbf{E}}^{n+2}\times {\mathbf{E}}^{n+2}
\,|\,
\vert x \vert = 1,\, \vert y\vert = 1, \, x\cdot y = 0 \big\},
\end{gather*}
the unit tangent bundle of~$S^{n+1}$.
The standard contact structure~$D$ of~$T_1S^{n+1}$
is def\/ined by the contact form
$\theta = \sum\limits_{i=0}^{n+1} y_idx_i$.
Then we have the double Legendrian f\/ibration
\begin{gather*}
\xymatrix{
& M=T_1S^{n+1} \ar[ld]_{\pi_1} \ar[rd]^{\pi_2}  &\\
W_1=S^{n+1} &  & W_2=S^{n+1}}
\end{gather*}
where
$\pi_1(x, y) = x$, $\pi_2(x, y) = y$.

\begin{Lem}
\label{sphere-M-A}
The above bi-Legendrian fibration is contactomorphic to the standard model:
The Lagrangian pair $(E_1, E_2) = ({\rm Ker}(\pi_2)_*, {\rm Ker}(\pi_1)_*)$
is flat.
\end{Lem}

\begin{proof}
For instance, on the open subset $U = \{ x_0 \not= 0, \,y_{n+1} \not= 0\}$
of $T_1S^{n+1}$, consider the system of coordinates
$x'_i = \frac{x_i}{x_0}$, $z' = - \frac{x_{n+1}}{x_0}$, $ p'_i =
\frac{y_i}{y_{n+1}}$,
$1 \leq i \leq n$. Moreover we set $\tilde{z}' = - \frac{y_0}{y_{n+1}}$.
Then $\tilde{z}' + z' - \sum\limits_{i=1}^n x'_ip'_i = 0$ is satisf\/ied.
Moreover we have
\begin{gather*}
\theta = - x_0y_{n+1}\left(dz' - \sum_{i=1}^n p'_idx'_i\right)
= x_0y_{n+1}\left(d\tilde{z}' - \sum_{i=1}^n x'_i dp'_i\right).
\end{gather*}
Thus we easily see that the
dif\/feomorphism $\Phi \colon  U \to {\mathbf{R}}^{2n+1}$ def\/ined by $\Phi(x, y)
= (x', z', p')$ provides the required contactomorphism.
\end{proof}

We set
\begin{gather*}
\Omega_1 =
c\ i_X(dx_0\wedge\cdots\wedge dx_{n+1}),
\end{gather*}
the standard volume form on $W_1=S^{n+1}$ multiplied by $c(\not= 0)\in {\mathbf{R}}$,
and set
\begin{gather*}
\Omega_2 =
i_Y(dy_0\wedge\cdots\wedge dy_{n+1}),
\end{gather*}
the
standard volume form on $W_2=S^{n+1}$.
Here $X = \sum\limits_{i=0}^{n+1} x_i\frac{\partial}{\partial x_i}$ and
$Y = \sum\limits_{i=0}^{n+1} y_i\frac{\partial}{\partial y_i}$.
The Reeb vector f\/ield $R$ on $T_1S^{n+1}$ is given
by
$R = \sum\limits_{i=0}^{n+1} \big(y_i\frac{\partial}{\partial x_i} -
x_i\frac{\partial}{\partial y_i}\big)$.
Then we have the following, which is the geometric form of
the equation ``Gaussian curvature $=c$''.

\begin{Prop}
The associated Monge--Amp{\`e}re system~$\mathcal{M}$
with Lagrangian pair on $M=T_1S^{n+1}$
is generated by~$\theta$ and
\begin{gather*}
\omega    =
i_R \big({\pi}_1^*{\Omega}_1-{\pi}_2^*{\Omega}_2\big)\\
\hphantom{\omega}{}  =
c \sum_{0 \leq j  < i \leq n+1} (-1)^{i+j} x_iy_j
dx_0 \wedge \cdots \wedge \breve{dx_j} \wedge \cdots \wedge
\breve{dx_i}
\wedge \cdots \wedge dx_{n+1}\\
\hphantom{\omega=}{} -  c \sum_{0 \leq i  <  k \leq n+1} (-1)^{i+k} x_iy_k
dx_0 \wedge \cdots \wedge \breve{dx_i} \wedge \cdots \wedge
\breve{dx_k}
\wedge \cdots \wedge dx_{n+1}
\\
\hphantom{\omega=}{} +
\sum_{0 \leq j  < i \leq n+1} (-1)^{i+j} y_ix_j
dy_0 \wedge \cdots \wedge \breve{dy_j} \wedge \cdots \wedge
\breve{dy_i}
\wedge \cdots \wedge dy_{n+1} \\
\hphantom{\omega=}{} -
\sum_{0 \leq i < k \leq n+1} (-1)^{i+k} y_ix_k
dy_0 \wedge \cdots \wedge \breve{dy_i} \wedge \cdots \wedge
\breve{dy_k}
\wedge \cdots \wedge dy_{n+1}.
\end{gather*}
The Gaussian curvature
of the projection to $W_1=S^{n+1}$ of a geometric solution
of~$\mathcal{M}$
is equal to constant $c$ outside of singular locus, while
the Gaussian curvature
of the projection to~$W_2=S^{n+1}$ of a geometric solution
of~$\mathcal{M}$
is equal to constant~$\frac{1}{c}$ outside of singular locus,
\end{Prop}

\begin{Rem}
Note that the Gaussian curvature $K$ of a hypersurface
in the unit sphere $S^{n+1}$ and its sectional curvature $S$
as a Riemannian manifold are related by
$S = K + 1$. For example, the great sphere $S^{n} \subset S^{n+1}$
has the constant Gauss map to $S^{n+1}$ and $K = 0$,
while has the sectional curvature $1$.
\end{Rem}

\begin{Rem}
The group $G = {\rm O}(n+2)$ acts on $W_1=S^{n+1}$, on $W_2=S^{n+1}$
 and on $S^{n+1}\times S^{n+1}$, thus on
$T_1S^{n+1} = \{ (x, y) \in S^{n+1}\times S^{n+1} \,|\, x\cdot y = 0\}$.
The contact structure $D = \{ \theta = 0\}$ and the contact form
$\theta = \sum\limits_{i=0}^{n+1} y_i dx_i$ are $G$-invariant, and
any contact form def\/ining $D$ is a non-zero constant multiple of~$\theta$.

We get the diagram:
\begin{gather*}
W_1=S^{n+1} = G/H' \leftarrow M=T_1S^{n+1} = G/H
\rightarrow W_2=S^{n+1} = G/H'',
\end{gather*}
for the isotropy groups $H$, $H'$ and $H''$ satisfying
\begin{gather*}
H' \cong {\rm O}(n+1) \hookleftarrow H \cong {\rm O}(n)
\hookrightarrow H'' \cong {\rm O}(n+1).
\end{gather*}
The $G$-invariant volume forms on $W_1=S^{n+1}$ and $W_2=S^{n+1}$
are unique up to non-zero constant.
We can construct
the Monge--Amp{\`e}re system $\mathcal{M}=\langle \theta, d\theta,
\omega \rangle $
with Lagrangian pair globally on $M=T_1S^{n+1}$.
\end{Rem}

{\bf 8.4.}
Monge--Amp{\`e}re system on $T_1H^{n+1}$
as $K=c$ in $H^{n+1}$.
In general, let us consider ${\mathbf{R}}_r^{n+2} = {\mathbf{R}}^{n+2}$
with the inner product
\begin{gather*}
x\cdot y = - \sum_{i=0}^{r-1} x_iy_i +
\sum_{j=r}^{n+1} x_jy_j,
\end{gather*}
of signature $(r, n+2-r)$.
We set, for $\varepsilon_1 = 0, \pm 1$, $\varepsilon_2 = 0, \pm 1$,
and for a~real number~$a$,
\begin{gather*}
S^{2n+1}_{\varepsilon_1, \varepsilon_2, a} =
\big\{ (x, y) \in {\mathbf{R}}_r^{n+2}\times {\mathbf{R}}_r^{n+2}
\,|\, x\cdot x = \varepsilon_1, \, y\cdot y = \varepsilon_2, \,
x\cdot y = a , \, x \not= 0, \, y \not= 0 \big\},
\end{gather*}
provided that $S^{2n+1}_{\varepsilon_1, \varepsilon_2, a} \not= \varnothing$.
Moreover we set $S^{n+1}_\varepsilon = \{ x \in {\mathbf{R}}_r^{n+2} \,|\, x\cdot x =
\varepsilon, \, x \not= 0 \}$ for $\varepsilon = 0, \pm 1$.
On $S^{2n+1}_{\varepsilon_1, \varepsilon_2, a}$, the contact structure $D$
is
def\/ined by
$\theta = - \sum\limits_{i=0}^{r-1} y_i dx_i +
\sum\limits_{j=r}^{n+1} y_j dx_j$. We have the double
Legendrian f\/ibration
\begin{gather*}
\xymatrix{
& M=S^{2n+1}_{\varepsilon_1, \varepsilon_2} \ar[ld]_{\pi_1} \ar[rd]^{\pi_2}  &\\
W_1 = S^{n+1}_{\varepsilon_1} &  & W_2 = S^{n+1}_{\varepsilon_2}}
\end{gather*}
by $\pi_1(x, y) = x$ and $\pi_2(x, y) = y$.

In the case where $r=1$, $\varepsilon_1=-1$, $\varepsilon_2=1$, $a=0$,
since
\begin{gather*}
S_{-1,1,0}^{2n+1}=T_1H^{n+1}=T_{-1}S_1^{n+1}\subset
{\mathbf{R}}_1^{n+2}\times {\mathbf{R}}_1^{n+2},\\
 S_{-1}^{n+1}=H^{n+1}\colon \ \text{the hyperbolic space},
\qquad
S_1^{n+1}\colon \ \text{the de Sitter space},
\end{gather*}
we have the double Legendrian f\/ibration
\begin{gather*}
\xymatrix{
& M=T_1H^{n+1}\cong H^{n+1}\times S^n \ar[ld]_{\pi_1} \ar[rd]^{\pi_2}  &\\
W_1 = H^{n+1} &  & W_2 =S_1^{n+1}}
\end{gather*}
(cf.\ the hyperbolic Gauss map~\cite{I-P-S}).

\begin{Lem}
\label{Minkowskii-M-A}
The above bi-Legendrian fibration is contactomorphic to the standard model.
The Lagrangian pair $(E_1, E_2) = ({\rm Ker}(\pi_2)_*, {\rm Ker}(\pi_1)_*)$
is flat.
\end{Lem}

\begin{proof}
For instance, on the open subset $\{ x_0 \not= 0,\,
y_{n+1} \not= 0\}$, we take the system of coordinates
$x'_i = \frac{x_i}{x_0}$, $p'_i = \frac{y_i}{y_{n+1}}$, $z' =
- \frac{x_{n+1}}{x_0}$.
Moreover we set $\tilde{z}' = \frac{y_0}{y_{n+1}}$.
Then we have $z' + \tilde{z}'  + \sum\limits_{i=1}^n x'_i p'_i = 0$, and
$\theta = - x_0y_{n+1}\Big(dz' - \sum\limits_{i=1}^n p'_i dx'_i\Big)$.
Then we see that the bi-Legendrian f\/ibration is contactomorphic to
the standard model.
\end{proof}

We set
\begin{gather*}
\Omega_1 =
c\ i_X(dx_0\wedge\cdots\wedge dx_{n+1}),
\end{gather*}
the standard volume form on $W_1=H^{n+1}$ multiplied by $c\in {\mathbf{R}}$ $(c\not= 0)$,
and set
\begin{gather*}
\Omega_2 =
i_Y(dy_0\wedge\cdots\wedge dy_{n+1}),
\end{gather*}
the
standard volume form on $W_2=S_1^{n+1}$.
Here $X = -x_0\frac{\partial}{\partial x_0}+\sum\limits_{i=1}^{n+1} x_i\frac{\partial}{\partial x_i}$
and
$Y = -y_0\frac{\partial}{\partial y_0}+\sum\limits_{i=1}^{n+1} y_i\frac{\partial}{\partial y_i}$.
The Reeb vector f\/ield~$R$ on~$T_1H^{n+1}$ is given
by
$R = \sum\limits_{i=0}^{n+1} \big(y_i\frac{\partial}{\partial x_i} +
x_i\frac{\partial}{\partial y_i}\big)$.
Then we have the following.

\begin{Prop}
The associated Monge--Amp{\`e}re system $\mathcal{M}$
with Lagrangian pair on $M=T_1H^{n+1}$
is generated by $\theta$ and
\begin{gather*}
\omega    =
i_R \big({\pi}_1^*{\Omega}_1-{\pi}_2^*{\Omega}_2\big)\\
\hphantom{\omega}{} =
c \sum_{i=1}^{n+1} (-1)^{i} x_0y_i
dx_1 \wedge \cdots \wedge \breve{dx_i} \wedge
\cdots \wedge dx_{n+1}
\\
\hphantom{\omega=}{}
+ c \sum_{0 \leq j  < i \leq n+1} (-1)^{i+j} x_iy_j
dx_0 \wedge \cdots \wedge \breve{dx_j} \wedge \cdots \wedge
\breve{dx_i}
\wedge \cdots \wedge dx_{n+1}
\\
\hphantom{\omega=}{} -  c \sum_{1 \leq i  <  k \leq n+1} (-1)^{i+k} x_iy_k
dx_0 \wedge \cdots \wedge \breve{dx_i} \wedge \cdots \wedge
\breve{dx_k}
\wedge \cdots \wedge dx_{n+1}
\\
\hphantom{\omega=}{}
+ \sum_{i=1}^{n+1} (-1)^{i} y_0x_i
dy_1 \wedge \cdots \wedge \breve{dy_i} \wedge
\cdots \wedge dy_{n+1}
\\
\hphantom{\omega=}{}  +
\sum_{0 \leq j  < i \leq n+1} (-1)^{i+j} y_ix_j
dy_0 \wedge \cdots \wedge \breve{dy_j} \wedge \cdots \wedge
\breve{dy_i}
\wedge \cdots \wedge dy_{n+1} \\
\hphantom{\omega=}{} -
\sum_{1 \leq i < k \leq n+1} (-1)^{i+k} y_ix_k
dy_0 \wedge \cdots \wedge \breve{dy_i} \wedge \cdots \wedge
\breve{dy_k}
\wedge \cdots \wedge dy_{n+1}.
\end{gather*}

The Gaussian curvature
of the projection to $W_1=H^{n+1}$ of a geometric solution
of~$\mathcal{M}$
is equal to constant $c$ outside of singular locus.
\end{Prop}

\begin{Rem}
The group $G = {\rm O}(1,n+1)$ acts on $W_1=H^{n+1}$, on $W_2=S_1^{n+1}$
and on $H^{n+1}\times S_1^{n+1}$, thus on
$T_1H^{n+1} = \{ (x, y) \in H^{n+1}\times S_1^{n+1} \,|\, x\cdot y = 0\}$.
The contact structure $D = \{ \theta = 0\}$ and the contact form
$\theta = - y_0 dx_0 + \sum\limits_{i=1}^{n+1} y_i dx_i$ are $G$-invariant, and
any contact form def\/ining $D$ is a non-zero constant multiple of $\theta$.

We get the diagram:
\begin{gather*}
W_1=H^{n+1} = G/H' \leftarrow M=T_1H^{n+1} = G/H
\rightarrow W_2=S_1^{n+1} = G/H'',
\end{gather*}
for the isotropy groups $H$, $H'$ and $H''$ satisfying
\begin{gather*}
H' \cong {\rm O}(n+1) \hookleftarrow H \cong {\rm O}(n)
\hookrightarrow H'' \cong {\rm O}(1,n).
\end{gather*}
The $G$-invariant volume forms on $W_1=H^{n+1}$ and $W_2=S_1^{n+1}$
are unique up to non-zero constant.
We can construct
the Monge--Amp{\`e}re system $\mathcal{M}=
\langle \theta, d\theta, \omega \rangle $
with Lagrangian pair globally on $M=T_1H^{n+1}$.
\end{Rem}

{\bf 8.5.}
The Monge--Amp{\`e}re systems introduced in
Sections~8.1--8.4
are all Euler--Lagrange Monge--Amp{\`e}re systems.

In fact, by Lemmas~\ref{Euclid-M-A}, \ref{sphere-M-A}, \ref{Minkowskii-M-A},
there exists a system of local coordinates
$(x', z', p') = (x'_1, \dots, x'_n, z', p'_1, \dots, p'_n)$
at any point of~$M$
such that a local contact form is given by
$\theta = dz' - \sum\limits_{i=1}^n p'_idx'_i$,
the Lagrangian pair is given by
\begin{gather*}
E_1   =    \{ v \in TM \,|\, \theta(v) = 0, \, dp'_1(v) = 0,
\dots, dp'_n(v) = 0 \},
\\
E_2  =  \{ u \in TM \,|\, \theta(u) = 0, \, dx'_1(u) = 0,
\dots, dx'_n(u) = 0 \},
\end{gather*}
the Monge--Amp{\`e}re system~${\mathcal M}$ is generated by
an $n$-form
\begin{gather*}
\tilde{\omega} =
f(x'_1, \dots, x'_n, z')dx'_1\wedge \cdots \wedge dx'_n -
g\left(p'_1, \dots, p'_n, \sum_{i=1}^n x'_ip'_i - z'\right)dp'_1\wedge \cdots
\wedge dp'_n
\end{gather*}
for some (pulled-back) functions $f$, $g$ on $W_1$, $W_2$ respectively.
If a local contactomorphism $\Phi \colon  M \to {\mathbf{R}}^{2n+1}$
and dif\/feomorphisms $\varphi \colon  W_1 \to {\mathbf{R}}^{n+1}$,
$\psi \colon  W_2 \to {\mathbf{R}}^{n+1}$ give the contactomorphism of
the bi-Legendrian f\/ibration to the standard model, then we have
\begin{gather*}
\big(\varphi^{-1}\big)^*\Omega_1 = f(x, z)dz \wedge dx_1 \wedge \cdots \wedge
dx_n,
\qquad
\big(\psi^{-1}\big)^*\Omega_2 = g(p, \tilde{z})d\tilde{z} \wedge dp_1
\wedge \cdots \wedge
dp_n,
\end{gather*}
for some non-zero functions $f$, $g$.

For example, we calculate $f$, $g$ explicitly in Euclidean geometry
(Section~8.2),
for the system of coordinates $x'_i = x_i$, $p'_i = - y_i/y_{n+1}$
$(1\leq i\leq n)$, $z' = x_{n+1}$ on the open set $\{ y_{n+1} \not= 0 \}$.
A~direct calculation yields
that the Monge--Amp{\`e}re system
for $K = c$ is generated by
\begin{gather*}
\tilde{\omega} = c \ dx'_1\wedge \cdots \wedge dx'_n
- (-1)^n\big(1 + {p'}_{1}^{2} + \cdots +
{p'}_{n}^{2}\big)^{-\frac{n+2}{2}}dp'_1\wedge \cdots \wedge dp'_n,
\end{gather*}
with the contact form $\theta = dz' - \sum\limits_{i=1}^n p'_idx'_i$.

Since $(\varphi^{-1})^*\Omega_1$ and $(\psi^{-1})^*\Omega_2$ are local volume forms
on ${\mathbf{R}}^{n+1}$, the argument in Example~\ref{standard model calcu} and Lemma~\ref{MA-system construction}
yields that
all Monge--Amp{\`e}re systems introduced in
Sections 8.1--8.4
are Euler--Lagrange Monge--Amp{\`e}re systems.

\subsection*{Acknowledgements}

{\sloppy The f\/irst author was partially supported by
Grants-in-Aid for Scientif\/ic Research No.~19654006.
The second author was partially supported by
Grants-in-Aid for Scientif\/ic Research (C) No.~18540105.
The authors would like to thank anonymous referees for the valuable
comments to improve the paper. 

}

\pdfbookmark[1]{References}{ref}
\LastPageEnding

\end{document}